\documentclass[11pt]{amsart}
\usepackage{amsmath}
\usepackage{amssymb}
\usepackage{amsfonts}
\usepackage{amsthm}
\usepackage{enumerate}
\usepackage[mathscr]{eucal}
\usepackage{amsthm}
\usepackage{amscd}

\newtheorem{theorem}{Theorem}[section]
\newtheorem{proposition}[theorem]{Proposition}
\newtheorem{conjecture}[theorem]{Conjecture}

\newtheorem{corollary}[theorem]{Corollary}
\newtheorem{lemma}[theorem]{Lemma}

\theoremstyle{definition}

\numberwithin{theorem}{section}
\numberwithin{definition}{section}
\numberwithin{equation}{section}

\def\tell{{\tilde \ell}}
\def\nuv{\nu}

\def\intslash{\rlap{\kern  .32em $\mspace {.5mu}\backslash$ }\int}
\def\inn#1#2{\langle#1,#2\rangle}

\newcommand{\Be}{\begin{equation}}
\newcommand{\Ee}{\end{equation}}

\newcommand{\Bm}{\begin{multline}}
\newcommand{\Em}{\end{multline}}
\newcommand{\Bea}{\begin{eqnarray}}
\newcommand{\Eea}{\end{eqnarray}}
\newcommand{\Beas}{\begin{eqnarray*}}
\newcommand{\Eeas}{\end{eqnarray*}}
\newcommand{\Benu}{\begin{enumerate}}
\newcommand{\Eenu}{\end{enumerate}}
\newcommand{\Bi}{\begin{itemize}}
\newcommand{\Ei}{\end{itemize}}

\newcommand{\e}{\varepsilon}
\newcommand{\R}{\mathbb{R}}

\renewcommand{\le}{\leqslant}
\renewcommand{\ge}{\geqslant}

\def\intslash{\rlap{\kern  .32em $\mspace {.5mu}\backslash$ }\int}
\def\qsl{{\rlap{\kern  .32em $\mspace {.5mu}\backslash$ }\int_{Q}}}
\def\Re{\operatorname{Re\,}}

\def\R{\mathbb R}

\def\emph#1{{\it #1 }}

\def\cf{{\it cf}}

\def\supp{{\text{\rm supp\,}}}

\def\inn#1#2{\langle#1,#2\rangle}

\def\noi{\noindent}

\def\meas{{\text{\rm meas}}}

\def\lc{\lesssim}
\def\gc{\gtrsim}

\def\eps{\varepsilon}

\def\la{\lambda}

\def\vphi{\varphi}

\def\fA{{\mathfrak {A}}}
\def\fB{{\mathfrak {B}}}

\def\fQ{{\mathfrak {Q}}}

\def\fZ{{\mathfrak {Z}}}

\def\fa{{\mathfrak {a}}}

\def\fc{{\mathfrak {c}}}

\def\fv{{\mathfrak {v}}}

\def\fz{{\mathfrak {z}}}

\def\bbR{{\mathbb {R}}}

\def\bbZ{{\mathbb {Z}}}

\def\cA{{\mathcal {A}}}

\def\cC{{\mathcal {C}}}

\def\cE{{\mathcal {E}}}
\def\cF{{\mathcal {F}}}

\def\cK{{\mathcal {K}}}

\def\cP{{\mathcal {P}}}
\def\cQ{{\mathcal {Q}}}

\def\e{\varepsilon}

\def\Re{\operatorname{Re\,} }

\renewcommand{\le}{\leqslant}
\renewcommand{\ge}{\geqslant}

\begin{document}

\newcommand{\w}{\widehat{f}}
\newcommand{\B}{\mathbb{B}}


\title[Space-time estimates for the Schr\"odinger operator]{On space-time
estimates for the Schr\"odinger operator}
\thanks{Supported in part by NRF grant 2012-0000940 (Korea),
ERC grant 277778, MICINN grant MTM2010-16518 (Spain) and NSF grant
0652890}

\author{Sanghyuk Lee \ \ \ \ Keith M. Rogers \ \ \ \ Andreas Seeger}


\address{Sanghyuk Lee\\ School of Mathematical Sciences,
Seoul National University, Seoul 151-742, Korea}
\email{shlee@math.snu.ac.kr}

\address{Keith Rogers \\
Instituto de Ciencias Matematicas
CSIC-UAM-UC3M-UCM\\
28049 Madrid, Spain} \email{keith.rogers@icmat.es}

\address{Andreas Seeger \\ Department of Mathematics \\ University of Wisconsin\\
480 Lincoln Drive\\
Madison, WI, 53706, USA} \email{seeger@math.wisc.edu}

\begin{abstract}
We prove mixed-norm space-time estimates for solutions of the
Schr\"odinger equation, with initial data in $L^p$-Sobolev or  Besov
spaces,
and clarify the relation with adjoint restriction.
\end{abstract}

\maketitle

\section{Introduction}\label{intro}

We are concerned with regularity questions for
the solution $u$ of the initial value problem for the Schr\"odinger equation
on $\bbR^d\times I$,
$$i\partial_tu +\Delta u=0, \qquad u(\,\cdot\,,0)=f,$$
where  $I$ is  a compact time interval.  When $f$ is a Schwartz
function, the solution can be written as $u=Uf$, with
\begin{equation}\label{form}
Uf(x,t)\equiv e^{it\Delta}\! f(x) = \frac{1}{(2\pi)^{d}}
\int_{\mathbb{R}^d} \widehat
f(\xi)\,e^{-it |\xi|^2+i\inn x\xi
} d\xi\,;\end{equation} here $\,\widehat{\,}\,$ denotes the Fourier transform  defined by
$\,\widehat{f}(\xi)\,=\int f(y)\,e^{-i\inn y\xi}dy $.

Bounds for the solution in  the spaces $L^r(I;L^q(\mathbb R^d))$
with initial data in $L^2$-Sobolev spaces  have been extensively
studied; these are known as  \lq Strichartz estimates' and they play
an  important role in the study of the  nonlinear  equation
 (see for example  \cite{tabook}).
In this paper we are  instead  concerned with bounds in the  spaces
$L^q(\bbR^d;L^r(I))$, equipped with the norm
$$ \|u\|_{L^q(\bbR^d;L^r(I))}= \Big(\int_{\bbR^d}\Big(\int_I |u(x,t)|^r
\, dt\Big)^{q/r}dx \Big)^{1/q}
$$
 when the initial datum is given in
Sobolev spaces $L^p_\alpha$,
with norm $\|f\|_{L^p_\alpha}=\|(I-\Delta)^{\alpha/2}f\|_{L^p(\bbR^d)}$. We thus
seek to prove the bound \begin{equation}\label{mixbes}
\|Uf\|_{L^{q}(\bbR^d;L^r(I))}\le C\|f\|_{L^p_{\alpha}(\bbR^d)}\,,
\end{equation}
for suitable choices of $p,q,r$ and $\alpha$. Unlike the estimates
in  $L^r(I;L^q(\mathbb R^d))$, the inequality \eqref{mixbes}  is no
longer invariant under Galilean transformations when $q\neq  r$
which usually  makes the problem   more difficult.

Estimates with particular $p,q$ and $r$ are
related to several well-known problems in harmonic analysis and various results have been  obtained in specific cases.  Notably, when $r=\infty$ and $p=2$,
   \eqref{mixbes} is the global version of the usual (local) maximal
estimates  which have been studied to prove
pointwise convergence of $Uf$
as $t\to 0$ (see for example \cite{camax}, \cite{ve}, \cite{boPrinc}, \cite{tava}, \cite{le1}). Two of the authors  \cite{rose}  proved  the sharp maximal estimates for some  $p=q>2$  which strengthen the fixed time estimates due to  Fefferman--Stein~\cite{fest}  and Miyachi~\cite{mi}.  When  $p=q> 2$ and $r=2$, the problem
 is closely related  to square function estimates for Bochner--Riesz
operators, and also to $L^q(L^2)$ regularity of solutions for the
wave equation   (see \cite{lerose} and \S \ref{secondbesch}).
Finally, for $p=2$,  some $q,r\in (2,\infty)$ and $I=\mathbb R$,
Planchon~\cite{planch} considered a homogeneous version of the
problem replacing $L^p_\alpha$ with the homogeneous space $\dot
H^\alpha$, see also  \cite{KPV}, \cite{RV}, \cite{tava} for closely
related results.   In this article we obtain some
new results on  \eqref{mixbes} for various choices of  $p, q$ and
$r$ and  clarify the relations with the aforementioned problems.

\subsection*{\it Connection with adjoint restriction estimates}  It is known that \eqref{mixbes} is closely related to estimates for  the adjoint restriction operator defined on a compact portion of the paraboloid in $\bbR^{d+1}$.   Various maximal and smoothing estimates were obtained by relying on the adjoint restriction estimate, or its bilinear and multilinear variants (see \cite{sj}, \cite{ve}, \cite{boPrinc}, \cite{tava},
\cite{le1}, \cite{ro1}, \cite{B}, \cite {bogu}).  Here we prove  an actual  equivalence of the space-time regularity estimates with  estimates for the adjoint restriction operator, which allows us to extend the range of \eqref{mixbes} by combining it with recent progress on the restriction problem \cite{bogu}. A related result establishing the equivalence between adjoint restriction and Bochner--Riesz for paraboloids was found by Carbery \cite{car}.


Let $\cE$ denote the adjoint restriction (or
Fourier extension) operator given by \Be\label{ext}\cE f(\xi, s )=
\int_{|y|\le 1} f(y)\, e^{i s |y|^2-i\inn{\xi}{y} }dy, \quad
(\xi, s )\in \bbR^{d}\times \bbR. \Ee

\noi {\bf Definition.} {\it We say that
 $\text{\rm R}^*(p\to q)$
holds if $\cE: L^p(\bbR^d)\to L^q(\bbR^{d+1})$ is bounded.}

\medskip

The  critical cases for adjoint restriction occur when
$q=\frac{d+2}{d}p'$, and for a given $q$ we denote the critical $p$
by $p(q)$. In that case, it follows from the explicit formula
\Be\label{expl}Uf(x,t)= \frac{1}{(4\pi it)^{d/2}}\int
\exp\Big(\frac{i|x-y|^2}{4t} \Big)f(y)\, dy\Ee and scaling that
 $\text{\rm R}^*(p(q)\to q)$  implies the
$L^{p(q)}(\Bbb R^d)\to L^{q}(\Bbb R^d\times I) $ boundedness of $U$.
Moreover it was shown in \cite{ro1} that it implies the
$L^{q}_\alpha\to L^{q}(\bbR^d\times I)$ bound for $\alpha>d(1
-\frac{2}{q})-\frac{2}{q}$. We strengthen the connection between
 $\text{\rm R}^*(p\to q)$
and Schr\"odinger estimates  by establishing an   equivalence for
general $p,q$. In order to  formulate it  we invoke
Besov spaces $B^p_{\alpha,\nu}$.
 Recall that
$\|f\|_{B^p_{\alpha,\nu}}=(\sum_{k\ge 0}2^{k\alpha\nu}\| P_kf\|_p^\nu)^{1/\nu}$
where  for $k\ge 1$, the
operators $P_k$  localize frequencies to annuli of width
$\approx 2^k$ and $P_0=I-\sum_{k\ge1} P_k$.

\begin{theorem} \label{besov-vs-ext}
Suppose $2\le p\le q<\infty$.
Then, for every $\nu\in (0,2]$,  the following are equivalent:

(i) $\text{ \rm R}^*(p\to q)$ holds.

(ii) The operator $U: B^p_{\alpha,\nu}(\bbR^d)
\to L^q(\bbR^{d}\times
I)$ is bounded with
 $\alpha= d\big(1-\frac 1p-\frac{1}{q}\big)-\frac{2}q.$
\end{theorem}
In \S \ref{adjressect} we shall also formulate more technical
variants of Theorem \ref{besov-vs-ext} which are valid for mixed
norm spaces.

We note that the restriction $\nu\le 2$ in Theorem
\ref{besov-vs-ext}  is only needed for the implication $(i)\! \Rightarrow \!(ii)$. Moreover, the theorem implies  that
 $\text{ \rm R}^*(p\to q)$  holds if and only if for all $\la>1$
the inequality
$$\|Uf\|_{L^q(\bbR^{d+1})} \lc \la^{ d(1-\frac 1p-\frac{1}{q})-\frac{2}q} \|f\|_p$$
holds for all $f\in L^p$ with frequency support  in $\{\xi:\la/2\le |\xi|\le
2\la\}$;  of course for those $f$ the parameter $\nu$ plays no role.
For more general initial data recall that   $B^p_{\alpha,\nu}$
is contained in the Sobolev space
$L^p_\alpha$ for $\nu\le \min \{2,p\}$, and vice versa,
$L^p_\alpha$ is contained in
$B^p_{\alpha,\nu}$
for $\nu\ge \max \{2,p\}$.
It remains open whether the condition $\nu\le 2$ is necessary  and
 whether $B^p_{\alpha,2}$ can be replaced
with $L^p_\alpha$ in Theorem \ref{besov-vs-ext}.
However it follows from a result in  \cite{lerose} (see \S\ref{cflp} below)
that if one is willing to give up an endpoint in the $q$-range
then one can also
obtain results on  larger spaces including  $L^p_\alpha$, as well as mixed norm inequalities with  $r\ge q$.

\begin{corollary}\label{SobCor2} Let $2<q_0<\infty$, $1\le p_0\le q_0$, and suppose that
$\text{\rm R}^*(p_0\to q_0)$ holds. Let $q_0<q<\infty$, $q\le r\le \infty$
and suppose that $0\le \frac{1}{p}-\frac{1}{q}\le
\frac{1}{p_0}-\frac{1}{q_0}$.
 Then $$
\|Uf\|_{L^q(\bbR^{d};L^r(I))}\le C\|f\|_{B^p_{\alpha,q}(\bbR^d)},
\qquad \alpha= d\big(1-\tfrac 1p-\tfrac 1q\big)-\tfrac 2r.
$$
\end{corollary}
By the  trivial $\text{R}^*(1\to\infty)$ estimate and interpolation
one can deduce the conclusion in the larger  range $p_1(q)<p\le q$,
where $p_1(q)<p_0$ is defined by $\frac{1}{p_1(q)}=
\frac{1}{p_0}+(1-\frac {q_0}{q})(1-\frac 1{p_0})$.
The
recent progress on $\text{\rm R}^*(p\to p)$
by Bourgain and
Guth \cite{bogu} can  be used to prove new estimates of the form
\begin{equation*}\label{mixbes2}
\|Uf\|_{L^{p}(\bbR^d;L^r(I))}\le
C\|f\|_{B^p_{\alpha,p}(\bbR^d)},\quad \alpha=d\big(1-\tfrac
{2}p\big)-\tfrac 2r.
\end{equation*}
In two spatial dimensions their result implies that the displayed
estimate holds in the case $r\ge p$ for $p\in(56/17,\infty)$ (see
\cite[pp. 1265]{bogu}); moreover, in higher dimensions, it holds for
the range $p\in(p_{BG}(d),\infty)$ with $p_{BG}(d)=2+
\frac{12}{4d+1-k} $ if $ d+1\equiv k ~(mod~ 3),\ k=-1, 0,1$. This
improves the result of \cite{rose}, where the estimate was shown to
hold in the range $p\in(\frac{2(d+3)}{d+1},\infty)$ using the
bilinear estimate of Tao~\cite{ta}. We will also see that Bourgain
and Guth's result can be combined with Tao's restriction bilinear
estimate to obtain the critical restriction estimates $\text{\rm R}^*(p(q)\to q)$  for
some range of
 $q$ with $q<\tfrac{2(d+3)}{d+1}$ (see
\S \ref{bgpq}).

\subsection*{\it Necessary conditions} We now consider  necessary conditions on $p,q,r$ and $\alpha$ for  \eqref{mixbes} to hold.   As previously mentioned,  due to connections with  other  problems,
 conditions  for  specific choices  of $p,q$ and $r$ are known, and
examples in those  special cases are also relevant when proving
necessary conditions for general  $p,q$ and $r$. However we also
establish additional   conditions which seem to have not been
noticed before. In particular the  necessity of the strict
inequalities in $(v)$, $(vi)$ in the following proposition
 are proved by constructions which involve the
Besicovich set (see \S\ref{necsect}).

In what follows we set
$\alpha_{cr}(p;q,r):= d(1-\frac 1p-\frac 1q)-\frac 2r$.

\begin{proposition}\label{necprop} Let $p,q,r\ge 2$ and
suppose that there is a constant $C$ such that
\begin{equation*}
\|Uf\|_{L^{q}(\bbR^d;L^r(I))}\le C\|f\|_{L^p_{\alpha}(\bbR^d)}\,,
\end{equation*}
holds whenever $f\in L^p_{\alpha}(\bbR^d)$.
Then
\begin{enumerate}[(i)]
\item  $p\le q$,
\item $ \alpha\ge \alpha_{cr}(p;q,r)$
\item $\alpha\ge \frac 1q-\frac 1r$,

\item $\alpha\ge  \frac 1q-\frac 1p$,


\item   $\alpha>  \frac 1q-\frac 1p$\ \ if\, $r>2$,

\item  $\alpha>  0$\ \  if\,  $r=2$, $p=q>2$, $d\ge 2$.

\end{enumerate}
The same conditions hold if we
replace
 Sobolev norm $L^p_{\alpha}$ by the
Besov norm  of
$B^p_{\alpha,\nu}$.
\end{proposition}

The condition $(i)$ is a simple consequence of translation invariance.
When  $p=2$, the condition $(ii)$ coincides with $(iii)$ if
$\frac{d+1}q+\frac 1r= \frac d2$. This is the condition in the
endpoint version of Planchon's conjecture (\cf. \cite{planch},
\cite{lerova}); that for these exponents $U:
\dot{H}^\alpha(\bbR^d)\to L^q(\R^d;L^r(\R))$ with
 $\alpha=d\big(\tfrac 12-\tfrac 1q\big)-\tfrac 2r$ and $r\ge 2$. If $p=2$ and $r=\infty$,  then the conditions $(iii)$  and $(v)$ follow from the necessary conditions for Carleson's problem \cite{camax,sj2}, via an equivalence between local and global estimates \cite{ro1}.

The necessary conditions also naturally  connect to those in the
restriction and Bochner--Riesz problems.
 The necessity of the condition  $(vi)$ in dimensions $d\ge 2$ comes  from
the fact that a sharp square function estimate for the Schr\"odinger
operator  implies sharp bounds on Bochner--Riesz multipliers.  When
$p=q$ and $2\le r\le q$, the condition $\alpha\ge
\alpha_{cr}(p;p,r)$
is more restrictive than
$(vi)$  if $d(\frac 12-\frac 1p)-\frac 1r>0$.
In particular, if  $r=2$ and
$\alpha= \alpha_{cr}(p;p,2)$,  by $(vi)$  the  range  $p>\frac{2d}{d-1}$
is necessary
(as can be deduced from the connection to the Bochner--Riesz
conjecture in $\bbR^d$),
and  for $r=p$, $\alpha= \alpha_{cr}(p;p,p)$
 the  range $p>\frac{2(d+1)}{d}$ is necessary
  (as can be deduced from the connection with the adjoint restriction problem  for the
paraboloid in $\bbR^{d+1}$, \cf. Theorem \ref{besov-vs-ext}).  On the other hand, if  $p<q$, $r=2$,  the condition $ \alpha\ge
\alpha_{cr} (p;q,2)$ is more restrictive than $(iv)$ if
$\frac{d+1}q\le \frac{d-1}{p'}$,  the familiar range for the adjoint
restriction theorem for  the sphere in $\bbR^d$. Likewise if,
$p<q=r$ then the condition $ \alpha\ge  \max\{0,\alpha_{cr}
(p;q,q)\}$ implies $\frac{d+2}q\le \frac{d}{p'}$, the range for the
adjoint restriction theorem for  the paraboloid in  $\bbR^{d+1}$.

\subsubsection*{Remark (added March 2012)} When $d\ge 5$, an additional  necessary condition  can be deduced from
Bourgain's recent lower bounds for the Schr\"odinger maximal
estimate. Precisely he showed that $ \|UP_{2k}
f\|_{L^{q}(B(0,1);L^\infty[0,2^{-2k}])}\le C 2^{2sk}\|P_{2k}f\|_{2}
$ holds for $q\ge 2$ only if $s\ge 1/2-1/d$.  By scaling this
implies that $ \|UP_{k} f\|_{L^{q}(\mathbb
R^d;L^\infty[0,1])}\le C 2^{2sk}2^{kd(1/q-1/2)}\|P_{k}f\|_{2} $ can only hold
if $s\ge 1/2-1/d$. By Sobolev imbedding this can be perturbed to
give a necessary condition $$\alpha\ge 1-\tfrac{2}{d}
-d\big(1-\tfrac1p-\tfrac 1q\big)-\tfrac2r$$ for  $p,q,r\ge 2$, which is
effective when $p,q$ are close 2 and $r$ is relatively large.

\subsection*{\it Results for $d=1$  and $d=2$}
In one and two spatial dimensions,  via more refined analysis based
on bilinear technology,  it is possible to obtain sharp estimates.
First  we state   precise bounds for frequency localized functions
in one spatial dimension.

%

\begin{theorem}\label{onedimp-prec}
For large $\la$, let
\begin{equation*}
\fA_\la(p;q,r)=\sup\Big\{ \|Uf \|_{L^q(\bbR; L^r(I))}\, :\,
\|f\|_p\le 1, \quad \supp \widehat f\subset\{\xi: \la/5 \le
|\xi|\le 15\la\}\Big\}.
\end{equation*} Then
for  $\la\gg 1$,
 the following norm equivalences hold:

(i) For  $2\le r\le p\le q\le \infty$,
$$\fA_\la(p;q,r)\,\approx\, \begin{cases}
\la^{1/q-1/p} [\log \la]^{1/2-1/r}
&\text{ if }\quad \frac 1q+\frac 1r \ge \frac 12\, ,\\
\la^{1-1/p-1/q-2/r} &\text{ if }\quad \frac 1q+\frac 1r < \frac 12\, .
\end{cases}$$

(ii) For   $2\le p<r\le q\le \infty$,
$$\fA_\la(p;q,r)\,\approx\, \begin{cases}
\la^{1/q-1/r}
&\text{ if }\quad \frac 2q+\frac 1r \ge 1-\frac 1p\, , \\
\la^{1-1/p-1/q-2/r} &\text{ if }\quad \frac 2q+\frac 1r < 1-\frac 1p\, .
\end{cases}$$
\end{theorem}


Again,  using the result in \S\ref{cflp} we obtain

\begin{corollary}\label{sobcor1} Suppose that $2\le r\le p\le q$, $\frac 1q+\frac 1r<\frac 12$,
or  $2\le p< r\le q$, $\frac 2q+\frac 1r<1-\frac 1p$.

\noindent Then $U: B^p_{\alpha,q}(\bbR)\to L^q(\R;L^r(I))$ is
bounded with $\alpha=1-\frac 1p-\frac 1q-\frac 2r$.
\end{corollary}

To compare these results, recall
that
$B^p_{\alpha,q_1}\subset B^p_{\alpha, q_2}$ for  $q_1<q_2$, that
$B^p_{\alpha,2}\subset L^p_\alpha\subset B^p_{\alpha,p}$ when $p\ge 2$, and that $B^p_{\alpha,p}$ is the same as the Sobolev--Slobodecki space
$W^{\alpha,p}$ when $0<\alpha<1$. In higher dimensions, if one singles out
the case  $p=q$,  one could hope to prove  the following
\begin{conjecture}\label{conj} Let $p\in[2,\infty)$, $r\in[2,\infty]$ satisfy
$\frac{d}{p}+\frac{1}{r}<\frac{d}{2}$  and
$\frac{2d+1}{p}+\frac{1}{r}< d$.
 Then $U: B^p_{\alpha,p}(\bbR^d)\to L^p(\R^d;L^r(I))$ is
bounded with $\alpha=d(1-\frac {2}p)-\frac 2r$.
\end{conjecture}

In \cite{lerose}, the conjecture was proven in the reduced range
$p\in(\frac{2(d+2)}{d},\infty)$, and for $d=1$ it was proven in the
range $p\in(4,\infty)$. In \cite{rose}, the conjecture was proven
for $p\in(\frac{2(d+3)}{d+1},\infty)$ with $r\ge p$ (see \cite{ro1}
for a nonendpoint version).

Theorem~\ref{onedimp-prec} also provides the negative part of the following corollary. The positive part was proven in \cite[Proposition
5.2]{lerose}.

\begin{corollary}\label{onedim}
Let $2\le p<\infty$.
 Then
$U: L^p(\bbR)\to L^p(\bbR;L^r(I))$
is bounded if and only if $r\le 2$.
\end{corollary}






In two dimensions we can improve on the previously known  range in
$p$
if  $r$ is large; 
this  is closely related  to  results on  maximal operators for
$L^2_\alpha$ functions  (\cf. \cite{planch}, \cite{le1}, \cite{ro}, \cite{lerova}).
\begin{theorem}
\label{pop1} Let $\frac {16}{5}<p<\infty$ and $4\le r\le\infty$.
Then $U: B^p_{\alpha,p}(\bbR^2)\to L^p(\bbR^{2};L^r(I))$ is bounded
with $\alpha=2\big(1-\frac 2p\big)-\frac 2r$.
\end{theorem}
The range in $r$   can be further improved for $16/5<p<4$, by
interpolating with the above mentioned $L^p(L^p(I))$ bounds for
$p>56/17$ (see \cite{bogu}) and  the $L^p(L^2(I))$ bounds
of~\cite{lerose} for $p>4$. Moreover one can obtain intermediate
$L^p_\alpha\to L^q(L^r(I))$ bounds with the critical $\alpha$ by
interpolating with the sharp $L^2\to L^q(L^r)$ bounds in
\cite{lerova}.



\medskip

{\it Organization of this  paper.}  In the following section, we
prove Theorem~\ref{besov-vs-ext} and related mixed norm results. In
\S \ref{necsect} we discuss necessary conditions to show Proposition
\ref{necprop} and the  lower bounds in Theorem \ref{onedimp-prec}.
The upper bounds are proven in \S \ref{d=1}. In \S \ref{fifth} we
detail how to combine the frequency localized pieces to obtain
estimates for Besov and Sobolev spaces, and in the final section we
prove Theorem~\ref{pop1}.

\medskip

{\it Notation.} By $m(D)$ we denote the convolution operator
with Fourier multiplier
$m$; that is to say $m(D)f=(m\widehat f\,)^\vee$.
For two nonnegative quantities $A$, $B$ the notation $A\lc B$ is used for $A\le CB$, with  some unspecified constant $C$.
We also use $A\approx B$ to indicate that $A\lc B$ and $B\lc A$.

\section{$L^p\to L^q(L^r(I))$
bounds and the adjoint restriction operator}
\label{adjressect}


We formulate a more technical version of Theorem \ref{besov-vs-ext}
that also applies to mixed-norm inequalities. In what follows let
\Be\label{annulus} \cA(\rho):=\{\xi\in \bbR^{d}: 3\rho\le|\xi|\le
12\rho\}\,. \Ee

\begin{theorem}\label{besovextequiv}
Let $p,q,r\in [2,\infty]$, $p\le q$,  $\beta> -d(\frac{1}{2}-\frac{1}{p})$
and $0<\nu\le 1$. Then
the inequality
\Be
\label{lqlrforext}
\sup_{\la> 1} \la^{-\beta} \sup_{\|f\|_p\le 1}
 \Big(\int_{\cA(\la)} \Big(\int_{\la}^{2\la}
  |\cE f(\tfrac{ s }{\la}\xi,  s )|^r d s  \Big)^{q/r} d\xi\Big)^{1/q}
<\infty\,
\Ee
holds if and only if for $\gamma = d(1-\frac 1p-\frac 1q)-\frac 2r+2\beta,
$
\Be\label{lqlrbesSchr}
\sup_{\|f\|_{B^p_{\gamma,\nu}} \le 1}
\Big\| \Big(\int_{-1}^1 |e^{it\Delta} f|^r dt\Big)^{1/r}\Big\|_q <\infty\,.
\quad
\Ee
If in addition $r<\infty$ this  equivalence remains valid for the range
$0<\nu\le 2$.

\end{theorem}

Taking Theorem \ref{besovextequiv} for granted we can quickly give

\begin{proof}[Proof of Theorem \ref{besov-vs-ext}]
 By Theorem
\ref{besovextequiv} we just have to show that $\text{\rm R}^*(p\to
q)$ is equivalent with \eqref{lqlrforext} for large $\la$, in the
case $q=r$ and $\beta=0$. Clearly the latter is implied by
$\text{\rm R}^*(p\to q)$; this follows by a change of variables
$(\eta, s )= ( s \lambda^{-1}\xi,  s )$ which has Jacobian
bounded above and below in the  region where $ s \approx\lambda$.

{\it Vice versa},  supposing that \eqref{lqlrforext} holds in the case $q=r$ and
$\beta=0$, by the change of variables, we have that $\cE: L^p(\bbR^d)\to
L^q(W_\lambda),$ where
$$W_\lambda=\{\,(\xi, s )\, :\,  s \in [\lambda,2\lambda],\quad x\in \cA( s )\,\}.$$ For
$\omega\in \bbR^{d+1}$ define $f^\omega(y)= e^{i \inn \omega y
-i\omega_{d+1}|y|^2} f(y)$ and observe that $\cE f^\omega=\cE
f(\cdot-\omega)$. Thus using  a finite number of translations we see
that
  $\cE: L^p(\bbR^d)\to L^q(B_\la)$,  where $B_\lambda$ is the ball in $\bbR^{d+1}$
of radius  $\lambda$ centred at the origin, and the operator norm
is uniformly bounded in $\lambda$. Letting $\lambda\to\infty$ yields
$\text{\rm R}^*(p\to q)$.
\end{proof}

We now proceed to prove Theorem \ref{besovextequiv}.

\begin{lemma} \label{resimpliesSchr} Let $p,q,r\in [2,\infty]$ with $p\le
q$ and let $\la\gg 1$.
Suppose that
 \Be\label{adjrestrassu}
\Big(\int_{\cA(\la^2)} \Big(\int_{\la^2}^{2\la^2}
  |\cE f(\tfrac{ s }{\la^2}\xi,  s )
|^r d s \Big)^{q/r} d\xi\Big)^{1/q} \le A\|f\|_p \Ee holds. Then, for  $\psi\in C^\infty_c$ with support in
$\{\xi:4<|\xi|<5\}$,
\begin{equation}\label{fixedfr}
\Big\|\Big(\int_{1/2}^1| e^{it\Delta} \psi(\tfrac D\la ) f|^r dt\Big)^{1/r}
 \Big\|_q\lc A  \la^{\alpha} \|f\|_{p}\,,\quad  \alpha = d\big(1-\tfrac 1p-\tfrac 1q\big)-\tfrac 2r.
\end{equation}

\end{lemma}

\begin{proof}[Proof]
If $f_\la$ is the characteristic function  of a ball of radius $(100
\la)^{-2}$ then $|\cE (f_\la)(\tfrac{ s }{\la^2}\xi,  s )|\ge
\la^{-2d}$ for $(\xi, s )\in \cA(\la^2)\times [\la^2,2\la^2]$. The
resulting lower bound $A\ge c \la^{2d(-1+ 1/p+ 1/q)+2/r}$ (which  is
far from being sharp) will be used repeatedly to dominate certain
error terms which decay fast in $\la$.

The convolution kernel for $e^{it\Delta}\psi(\tfrac D\la )$ can be written
as
\begin{align*}K^\la_t(x)
\,=\,\Big(\frac{\la}{2\pi}\Big)^d \int \psi(\xi)\, e^{-it\la^2|\xi|^2+i\la\inn x\xi} d\xi.
\end{align*}
By integration by parts it follows that   \Be \label{offexcset}
|K^\la_t(x)| \le C_N|x|^{-N},  \quad\text{for $|x|\ge 11\la$.} \Ee
Hence, by a standard argument, \eqref{fixedfr} reduces to showing the
inequality \Be\label{lqlr2}\Big(\int_{|x|\le 11 \la}
\Big(\int_{1/2}^1 |K^\la_t*f|^r dt\Big)^{q/r} dx\Big)^{1/q}
 \lc A \la^{\alpha} \|f\|_p, \quad  \alpha =  d\big(1-\tfrac 1p-\tfrac 1q)-\tfrac 2r
\Ee for $f$ supported in the cube of sidelength $\la(2d)^{-1}$
centred at the origin. Indeed, suppose  that \eqref{lqlr2} is
verified, let $\fQ_\la=\{Q\}$ be a grid of cubes with
 sidelength $\la(2d)^{-1}$, and centres $x_Q$, and let $B_Q$ be the ball of radius $11\la$ centred at $x_Q$.
Then we may estimate the $L^q(\mathbb R^d; L^r([1/2,1]))$ norm of
$e^{it\Delta}\psi(\tfrac D\la )$ by
 \Be \label{twoterms}
\begin{aligned}
&\Big(\int \sum_Q \chi_{B_Q}(x) \Big(\int_{1/2}^1
|K^\la_t*[f\chi_Q](x)|^r dt\Big)^{q/r} dx\Big)^{1/q}
\\ &\qquad+
\Big(\int \sum_Q \chi_Q(x) \Big(\int_{1/2}^1
|K^\la_t*[f\chi_{\bbR^d\setminus B_Q}](x)|^r dt\Big)^{q/r}
dx\Big)^{1/q}
\end{aligned}
\Ee by Minkowski's inequality in $L^r$. We use  the finite overlap
of the balls, the translation invariance of the operators  and
\eqref{lqlr2} to estimate
the first term  by 
\begin{align*}
CA\la^\alpha\Big(\sum_Q\|f\chi_Q\|_p^q\Big)^{1/q} \lc CA\la^\alpha
\|f\|_p
\end{align*}
where for the last inequality we have used the assumption $p\le q$.
For the second term in \eqref{twoterms} we use \eqref{offexcset}
with $N>2d$ and then Young's inequality
 to bound it by
\begin{align*}&
C\Big(\int \Big[
 \int_{|w|\ge 10\la}
|w|^{-N}|f(x-w)|dw\Big]^q dx\Big)^{1/q} \lc \la^{-N+d(1-\frac
1p+\frac 1q)}\|f\|_p \lc A\la^{\alpha}\|f\|_p.
\end{align*}
We used the trivial  lower bound for $A$ in the last step.

Our task is now to prove \eqref{lqlr2}.
We use
a   stationary phase calculation
to see that
$K^\la_t= H^\la_t+E_t^\la$, where
$$H^\la_t(x) =  \frac{e^{-i|x|^2/4t}}{(4\pi it)^{d/2}}
\sum_{\nu=0}^M\psi_\nu\big(\frac{x}{2\la t}
\big)\la^{-\nu}$$
and $$|E_\la(x,t)|\le C_L \la^{-L}$$
where we choose $L\gg d$. For the leading term $\psi_0=\psi$, and  the functions $\psi_\nu$
are obtained by letting certain differential operators act on $\psi$; thus $\psi_\nu(w)=0$ for $|w|\le 4$ and
$|w|\ge 5$.

For the error term we use a trivial bound
$$\Big(\int_{|x|\le 11 \la} \Big(\int_{1/2}^1
\Big[\int |E_\la(x-y,t)| |f(y)| \,dy\Big]^rdt\Big)^{q/r}dx\Big)^{1/q} \lc
 \la^{d-L} \|f\|_p\,\lc\, A\la^\alpha \|f\|_p\,.
$$ For the oscillatory terms we have to  prove the inequality
\Be\label{lplp3}\Big( \int_{|x|\le 11 \la} \Big(\int_{1/2}^1
\Big|\int \psi_\nu \big(\frac{x-y}{2\la t}\big)
\exp\big(i\frac{|x-y|^2}{4t}\big) f(y) \,dy\Big|^rdt\Big)^{q/r}
dx\Big)^{1/q} \lc A \la^{\alpha}   \|f\|_p \Ee whenever $f$ is supported in
$\{|y|\le \la/2\}$. By a change of variable $t\mapsto u=1/t$ (with
$u\approx t\approx 1$) and the support properties for $\psi_\nu$
this  follows from
\begin{align}\label{lplp3b}\Big(\int_{\frac 72\la\le |x|\le \frac{21}{2} \la}\Big( \int_{1}^2
\Big|\int_{|y|\le \la/2} \psi_\nu\big(\frac{u(x-y)}{2\la}\big)
&\exp\big(i\frac u 4 (|y|^2-2\inn xy)\big)
\\&\times f(y) dy\Big|^{r}du\Big)^{q/r}dx \Big)^{1/q}
 \lc
A \la^{\alpha}   \|f\|_p \nonumber
\end{align}
whenever $f$ is supported in
$\{|y|\le \la/2\}$. We now use a parabolic
scaling in the $(x,u)$-variables  and set
$$x= \la^{-1}w,\quad u= \la^{-2} s;  \qquad y= 2\la z.$$
The
previous inequality becomes
\begin{multline}
\label{lplp3d}
\Big(\int_{\frac72\la^2\le |w|\le \frac{21}{2} \la^2} \Big(\int_{\la^2}^{2\la^2}
\Big|\int_{|z|\le 1} \psi_\nu\big(\frac{sw-2\la^2sz}{2\la^4}\big)
\\ \times\exp(i (s|z|^2-\inn {\frac{sw}{\la^2}}{z})) f(2\la z) (2\la)^d\,dz\Big|^r
\frac{ds} {\la^{2}}
\Big)^{q/r} \frac{dw}{\la^{d}}\Big)^{1/q}  \lc A \la^{\alpha}  \|f\|_p.
\end{multline}
We have the  Fourier series expansion
$\psi_\nu(x) =\sum_{\ell\in \bbZ^d} c_{\ell,\nu} e^{i\inn \ell x}$
for
$x\in [-\frac{9}{10}\pi,\frac 9{10}\pi]^d$ and for each $\nu$  the Fourier coefficients are rapidly decaying,
$|c_{\ell,\nu}|\le C_{N,\nu} (1+|\ell|)^{-N}$. Thus
$$
\psi_\nu\big(\frac{sw-2\la^2sz}{2\la^4}\big)
 =
\sum_\ell c_{\ell,\nu} e^{i\la^{-4}\inn {sw}\ell/2}
e^{-i \la^{-2} s\inn z\ell}.$$

Using Minkowski's inequality for the sum   and the rapid decay of the Fourier coefficients
 the previous inequality \eqref{lplp3b} follows
  from
\begin{align}
\label{lplp3e} \Big(\int_{\frac72\la^2\le |w|\le \frac{21}{2} \la^2}
\Big(\int_{\la^2}^{2\la^2} &\Big|\int_{|z|\le 1} \exp(i (s|z|^2-\inn
{\frac{s(w+\ell)}{\la^2}}{z} )) f(2\la z) \,dz\Big|^r {ds}
\Big)^{q/r} dw\Big)^{1/q} \\ &\lc (1+|\ell|)^M A\la^{\alpha-d+\frac
2r+\frac dq}\|f\|_p. \nonumber
\end{align}
The left hand side is trivially bounded by $C \la^{2/r+2d/q}$ and
therefore the displayed inequality holds for  $|\ell|\ge \la^2/4$.
If $|\ell|\le \la^2/4$, we change variables and
 see that for  \eqref{lplp3e}  we only need to show
\begin{align*}
\Big(\int_{3\la^2\le |w|\le 11 \la^2}
\Big(\int_{\la^2}^{2\la^2} \Big|\int_{|z|\le 1} \exp(i (s|z|^2-\inn
{\frac{sw}{\la^2}}{z} )) &g(z) \,dz\Big|^r {ds} \Big)^{q/r}
dw\Big)^{1/q} \\ &\lc A\la^{\alpha-d+\frac 2r+\frac
dq}\la^{d/p}\|g\|_p. \nonumber
\end{align*}
The right hand side is just $ A \|g\|_p$, so that this would  follow from
\eqref{adjrestrassu}.
\end{proof}

\begin{lemma} \label{Schrimpliesrestr} Let $p,q,r\in [2,\infty]$ and $\lambda\gg 1$.
 Let $2<a_0<a_1$ and let $\eta$ be a radial $C^\infty_c$ function which satisfies $\eta(\xi)=1$ for
$\frac{a_0-2}{4}\le |\xi|\le 2(a_1+2)$.
  Suppose
\Be\label{lqlrschr2}
\sup_{\|f\|_p\le 1} \Big\| \Big(\int_{1/2}^1 |e^{it\Delta}\eta(\tfrac D\la ) f|^r dt\Big)^{1/r}\Big\|_q \le
B\,.
\Ee
Then
\Be
\label{lqlrext2}
\Big(\int_{a_0\la^2\le |\xi|\le a_1\la^2}\Big(\int_{\la^2}^{2\la^2}
\big|\cE f \big(\tfrac{ s }{\la^2} \xi,  s \big)\big|^r
d s \Big)^{q/r}d\xi\Big)^{1/q}\lc B \la^{-d+\frac dp+\frac dq+\frac
2r } \|f\|_p. 
\Ee
\end{lemma}

\begin{proof}
In what follows let $\alpha= d(1-\frac 1p -\frac 1q)-\frac 2r$. We
begin by observing the lower bound $B\ge c \la^\alpha$ which follows
from the example in \S\ref{secondneccond}.

By a change of variable $\xi=\la x$, $ s = \la^2 \rho$, $y= 2\la z$
we see that
 \eqref{lqlrext2}  is equivalent with
\begin{multline*}
\qquad\Big(\int_{a_0\la\le |x|\le a_1\la}\Big(\int_1^2
\Big|\int_{|y|\le 2\la} f(\tfrac {y}{2\la})
e^{i(\rho|y|^2/4-\rho\inn xy/2)} dy\Big|^2  d\rho\Big)^{q/r}
dx\Big)^{1/q}
\\
\le CB\la^{-\alpha}(2\la)^d
 \la^{ -d/q-2/r}
\|f\|_p .\qquad
\end{multline*}
By inverting $t=1/\rho$ the previous inequality
follows from
\begin{multline*}
\qquad \Big(\int_{a_0\la\le |x|\le a_1\la}\Big(\int_{1/2}^1
\Big|\frac{1}{(4\pi i t)^{d/2}}\int_{|y|\le 2\la} g(y) e^{\frac{i|x-y|^2}{4t}}
dy\Big|^r  dt\Big)^{q/r} dx\Big)^{1/q}
\\
\lc B\la^{-\alpha} \la^{d-d/p-2/r}  \la^{-d/p}\|g\|_p \qquad
\end{multline*}
which can be rewritten as \Be \label{Ugbound}\Big(\int_{a_0\la\le
|x|\le a_1\la}\Big(\int_{1/2}^1 |e^{it\Delta}g(x)|^r
dt\Big)^{q/r} dx\Big)^{1/q} \lc B  \|g\|_p\,, \Ee for $g$ supported
in $\{y:|y|\le 2\la\}$. By assumption
$$\Big(\int_{a_0\la\le |x|\le a_1\la}\Big(\int_{1/2}^1
\Big| e^{it\Delta}\eta(\tfrac D\la )g(x) \Big|^r  dt\Big)^{q/r}
dx\Big)^{1/q} \le B \|g\|_p
$$
and  thus \eqref{lqlrext2} follows from the straightforward estimate
\Be\label{straightf} \Big(\int_{a_0\la\le |x|\le
a_1\la}\Big(\int_{1/2}^1 \Big|e^{it\Delta}(I-\eta(\tfrac D\la )) g(x)
\Big|^r  dt\Big)^{q/r} dx\Big)^{1/q} \le C_M \la^{-M} \|g\|_p\, ,
\Ee whenever $g$ is supported in $\{y:|y|\le 2\la\}$.

To see \eqref{straightf} we decompose the multiplier.
Let $\chi_0$ be smooth and supported in $\{|\xi|<2\}$ and $\chi_0(\xi)=1$ for $|\xi|\le 1$, and let
$\chi_k(\xi)=\chi_0(2^{-k}\xi)-\chi_0(2^{1-k}\xi)$, for $k\ge 1$. Let
$$E_{\la,k}(x,t)= \frac{1}{(2\pi)^d}\int \chi_k(\tfrac{\xi}{\la})(1-\eta(\tfrac{\xi}{\la})) e^{-it|\xi|^2+i\inn x\xi} d\xi 
$$
and  we need to bound
the expression
 $$\big(I-\eta(\tfrac D\la )\big)e^{it\Delta} g(x,t)=\sum_{k\ge 0} \int_{|y|\le 2\la} E_{\la,k}(x-y)
g(y) dy.$$
We now examine $\nabla_\xi(\inn {x-y}{\xi} -t|\xi|^2)= x-y-2t\xi $.
Since $a_0>2$,  for the relevant choices $a_0|\la|\le |x|\le
a_1\la$, $1/2\le t\le 1$, $|y|\le 2\la$ we have
$$| x-y-2t\xi|\ge \begin{cases} \frac 12(a_0-2)\la &\text{  if
$|\xi|\le \tfrac{a_0-2}{4}\la$, }
\\
\max\{ \tfrac{|\xi|}{2},\,\, (a_1+2)\la\}
&\text{  if
$|\xi|\ge (a_1+2)\la$.}\end{cases}
$$
Since
 $1-\eta(\la^{-1} \xi)=0$ for
$\frac{a_0-2}{4} \le |\xi|\le 2(a_1+2)$, after an $N$-fold
integration by  parts we find that $|E_{\la,k}(x-y,t) | \le
C_{N}(2^k\la)^{d-N}$ for this choice of $x,y,t$, and the estimate
\eqref{straightf} follows.
\end{proof}
To complete the proof of Theorem \ref{besovextequiv} we also need the following scaling lemma.
\begin{lemma}\label{smallerintervals}
 Let $\gamma>d(\frac 1p-\frac 1q)-\frac 2r$.
Suppose that for  $\la\gg1$ \Be\label{lplq1} \Big\| \Big(\int_{1/2}^1 |
 e^{it\Delta}\chi(\tfrac D\la) f|^r dt\Big)^{1/q}\Big\|_q \lc \la^\gamma
\|f\|_p, \Ee
where $\chi\in C^\infty_c  $ is supported in  $(1/2,2)$ (with
suitable bounds). Then, for  $\la\gg1,$ \Be\label{lplqga} \Big\|
\Big(\int_{I} | e^{it\Delta}\chi(\tfrac D\la) f|^r
dt\Big)^{1/r}\Big\|_q \lc \la^\gamma\|f\|_{p}. \Ee
\end{lemma}

\begin{proof}
 It is easy to calculate that
$$\sup_{0\le t\le (8\la)^{-2}}
|\cF^{-1} [\chi(\tfrac \cdot\la) \exp(-it|\cdot|^2)] (x)| \le C_N \la^d(1+\la|x|)^{-N}
$$ and thus, by Young's inequality,
\begin{align}
\notag\Big\| \Big(\int_0^{(8\la)^{-2}} |
e^{it\Delta}\chi(\tfrac D\la) f|^r dt\Big)^{1/r}\Big\|_q
&\lc \Big\| \la^{-2/r}\int \la^{d}(1+\la|y|)^{-N}|f(\cdot-y)|dy\Big\|_q\\ &\lc
\la^{d(\frac 1p-\frac 1q)-\frac 2r}
\|f\|_p.
\label{firstpiece}
\end{align}
Now letting $(8\la)^{-2}\le b\le 1$,
$$\Big(\int_{b/2}^b |e^{it\Delta }\chi(\tfrac{D}{\la})  f (x)|^r  dt\Big)^{1/r}
=b^{1/r}\Big(\int_{1/2}^1 \Big|\chi(\tfrac{D}{b^{1/2}\la}) e^{is\Delta } [f(b^{1/2}\cdot)]
(b^{-1/2}x)\Big|^r  ds\Big)^{1/r}.$$
Thus by  a change of variable \eqref{lplq1} implies
$$\Big\|\Big(\int_{b/2}^b |e^{it\Delta }\chi(\tfrac{D}{\la})  f |^r  dt\Big)^{1/r}\Big\|_q
\lc (\sqrt b)^{-d(\frac 1p-\frac 1q)+\frac 2r}  (\la\sqrt b)^{\gamma} \|f\|_p.
$$
We choose $b= 2^{-l}$,  and since $\gamma> d(\frac{1}{p}-\frac1q)-\frac2r$ we may
sum over $l$ with $(8\la)^{-2}\le 2^{-l}\le 1$ and combine with
\eqref{firstpiece}. Hence we get
$$ \Big\|
\Big(\int_0^1 |e^{it\Delta }\chi(\tfrac{D}{\la})  f|^r
dt\Big)^{1/r}\Big\|_q \lc \la^\gamma\|f\|_{p}. $$
Now $\eqref{lplqga}$ with $I=[-1,1]$ follows using the formula
$e^{-it\Delta } f= \overline{e^{it\Delta }\overline {f\,}}\,,$ and the triangle inequality. Finally, by scaling,  we can enlarge the time interval (so that the implicit constant is of course dependent on  the interval), and we are done.
\end{proof}

\begin{proof}[Proof of Theorem \ref{besovextequiv}]
The implication \eqref{lqlrbesSchr} $\Rightarrow$
\eqref{lqlrforext}, for all $\nu>0$, follows from
Lemma~\ref{Schrimpliesrestr}.

For the implication \eqref{lqlrforext} $\Rightarrow$
\eqref{lqlrbesSchr} we decompose $f=\sum_{k=0}^\infty P_k f$, with
the standard inhomogeneous decomposition, and assume for $k>1$  that
$\supp \widehat{P_k f}$ is contained in $\{\xi:2^{k-1}\le |\xi|\le
2^{k+1}\}$ and $\supp \widehat {P_0 f} $ is contained in  $\{\xi:
|\xi|\le 2\}$. We  estimate
 $\chi(t)  UP_k f(x,t)$ where
 $\chi\in C^\infty_c$ with $\chi(t)=1$ on $[-1,1]$.
Let $\widetilde P_k$ have similar properties to $P_k$, with $\widetilde P_k P_k=P_k$.  We prove  the  inequality
\Be \label{dyadicpiecesbesov}
\Big\| \Big(\int |\chi(t) U\widetilde P_k f(\cdot, t)|^r dt\Big)^{1/r}\Big\|_q
\lc   2^{k\gamma}\|f\|_p, \quad
\gamma=d\big(1-\tfrac1p-\tfrac 1q\big)-\tfrac 2r+2\beta\,,
\Ee
which we apply with $P_kf$ in place of $f$.
Now if $\beta>-d(1/2-1/p)$
then the restriction on $\gamma$ in Lemma \ref{smallerintervals} is
satisfied. Thus
\eqref{dyadicpiecesbesov} follows by combining
 Lemmata \ref{resimpliesSchr} and
\ref{smallerintervals} (together with a  finite decomposition and mild
rescaling).
This immediately yields the  implication \eqref{lqlrforext} $\Rightarrow$
\eqref{lqlrbesSchr}  in the range $\nu\le 1$.

If $r<\infty$ we can use
Littlewood--Paley theory  to extend this implication to  the case $\nu= 2$
(which implies the corresponding inequality for  $\nu<2$).
Let,   for a function $g$ on $\bbR^d\times \bbR$,
$$ R_{2k}g(x,t)= \frac{1}{2\pi} \int \int\beta(2^{-2k}\tau) e^{i\tau(t-s)} d\tau\,g(x,s) \, ds,$$
where  $\beta $ is supported in $[1/10, 10]$ and  $\beta(\tau)=1$ for
$\tau\in[1/8, 8]$.

The contribution $(I-R_{2k})[\chi  \,UP_k f]$ is negligible.  To see this one
 uses various standard integration by parts arguments,  in particular
the decay of $\int \chi(s)e^{is(|\xi|^2-\tau)} ds$ when $|\xi|^2 \gg \tau$
or $\tau\gg |\xi|^2$ to analyze the kernel. We omit the details which give
$$\Big\|\Big(\int_\bbR | (I-R_{2k})[\chi U P_k f]|^r dt\Big)^{1/r}\Big\|_{L^q(\bbR^d)}\lc C_N 2^{-kN} \|P_k f\|_p.$$
It thus  remains  to show
\Be \label{LPforR}
\Big\|\Big(\int_\bbR \Big| \sum_{k\ge 1} R_{2k}[\chi \,U P_k f]\Big|^r dt\Big)^{1/r}\Big\|_{L^q(\bbR^d)}\lc  \Big(\sum_k \big[ 2^{k\gamma}\|P_k f\|_p\big]^2\Big)^{1/2}.
\Ee

Using Littlewood--Paley theory on $L^r(\bbR)$  followed by  applications of
the triangle inequa\-lities for  $L^{r/2}$ and  $L^{q/2}$ we see that the left hand side of \eqref{LPforR} is controlled by a constant times
$$\begin{aligned}
&\Big\|\Big( \sum_k | \chi\, UP_k f|^2\Big)^{1/2}\Big\|_{L^{q}(\bbR^{d}; L^{r}(\bbR))}\,=\,\Big\| \sum_k \big| \chi\, UP_k f|^2\Big\|_{L^{q/2}(\bbR^{d}; L^{r/2}(\bbR))}^{1/2}
\\
&\le \Big(\sum_k \big \||\chi \,UP_k f|^2\big\|_{L^{q/2}(\bbR^{d}; L^{r/2}(\bbR))}\Big)^{1/2}\,= \Big(\sum_k \big \|\chi \,UP_k f\big\|_{L^{q}(\bbR^{d}; L^{r}(\bbR))}^2\Big)^{1/2}\,.
\end{aligned}
$$
Now  \eqref{LPforR}
 follows from \eqref{dyadicpiecesbesov}.
\end{proof}

\section{Proof of Proposition~\ref{necprop}}\label{necsect}

First we prove the easier necessary conditions
$(i)$-$(iv)$.

\subsection{\it The condition  $p\le q$} \label{firstneccond}
This  follows from the translation invariance
(see an argument in \cite{hoer1}).
More precisely, the $ L^p_\alpha(\bbR^d)\to L^{q}(\bbR^d;L^r(I))$ boundedness is equivalent
with the
$ L^p(\bbR^d)\to L^{q}(\bbR^d;L^r(I))$ boundedness  of the operator
$U[(I-\Delta)^{\alpha/2} f]$ which commutes with translations on $\Bbb R^d$.
Let
$A=\sup_{\|f\|_p\le 1}\| U[(I-\Delta)^{\alpha/2} f]\|_{L^q(L^r)}.$ Then by the density argument, for $\epsilon>0$  there is a $g\in C_c^\infty (\mathbb R^d)$  such that $A-\epsilon< \| U[(I-\Delta)^{\alpha/2} g]\|_{L^q(L^r)}$ and $\|g\|_p=1$.
One may test the inequality with $f= g+ g(\cdot+a e_1)$. Letting $a\to\infty$, we see that $(A-\epsilon)2^{1/q} \le A 2^{1/p}$, which gives $A2^{1/q} \le A 2^{1/p}$ by letting $\epsilon\to 0$, and thus $p\le q$.

\subsection{\it The condition  $\alpha\ge d(1-\frac 1p-\frac 1q)-\frac 2r$}\label{secondneccond}
 This condition follows by a focusing example (see for example \cite{ro1}). Let $\eta\in C_c^\infty$ be radial and  supported in $\{\xi: 1<|\xi|<2\}$.
Define for $\la\gg1 $, the function $f_\la$ by $\widehat {f_\la}(\xi) =
e^{i\frac{1}{2}|\xi|^2} \eta(\la^{-1}\xi)$. Then $\|f_\la\|_{L^p_\alpha} \lc \la^{\alpha+d/p}$.
  Moreover $|Uf(x,t)|\gc \la^d$
if, for suitable $c>0$,
 $|x|\le c\la^{-1}$ and $|t-\frac{1}{2}|\le c\la^{-2}$.
For large $\la$ this leads to the restriction
 $\alpha\ge d(1-\frac 1p-\frac 1q)-\frac 2r$.

\subsection{\it The condition   $\alpha\ge \frac1q-\frac1r$}
Let $g_\la$ be defined by $\widehat g_\la(\xi)= \chi(|\xi-\la e_1|)$, $\chi$ supported in an $\eps$-neighborhood of $0$
(see  \cite{dake}, \cite{rose}), so that $\|g_\la\|_{L^p_\alpha} \lc \la^\alpha$.
Also, $$Ug_\la(x,t)= \frac{1}{(2\pi)^d}\int \chi(|h|) e^{i\phi_\la(x,t,h)}\,dh$$ where
$\phi_\la(x,t,h)=-t|h|^2-t\la^2+x_1\la +\inn {x-2t\la e_1}{h}$.
Then $|Ug_\la(x,t)|\ge c_0>0$ if $|t-(2\la)^{-1}x_1|\le c\la^{-1} $ for $0\le x_1\le \la$,
$|x_i|\le c$, $i=2,\dots, d$.
It follows that $\|Uf\|_{L^q(L^r(I))}\ge \la ^{1/q-1/r}$. Hence the condition
$\alpha\ge 1/q-1/r$ follows.

\subsection{ \it The condition   $\alpha\ge \frac1q-\frac1p$}
 Let $\lambda\gg 1$ and set $\widehat{h_\lambda}(\eta)=\phi(|\eta'|)\lambda\phi(\lambda(\eta_1-\lambda))$ with $\phi\in C_c^\infty(\R)$.  Then $\|h_\lambda\|_{L^p_{\alpha}}\lesssim \lambda^\alpha \lambda^{1/p}$. Note that
\[ Uh_\lambda(x,t)=\frac{1}{(2\pi)^d}\int e^{-it|\eta'|^2+i\inn {x'}{\eta'})} \phi(|\eta'|) d\eta'  e^{-i\lambda^2t+i\lambda x_1}
\int e^{i(-t\xi_1^2-2\lambda t\xi_1+x_1\xi_1)}
\lambda\phi(\lambda\xi_1) d\xi_1,\] so that $|Uh_\lambda(x,t)|\ge
c>0$ if $|t|, |x'|\le c$ and $|x_1|\le c\lambda$ for small enough
$c>0$. This shows the necessity of $\alpha\ge 1/q-1/p$.

\

To show the conditions $(v)$ and  $(vi)$,  we use sharp bounds in the
construction of  Besicovich sets \cite{keich} and adapt Fefferman's  argument
for the  disc multiplier \cite{F} (see also \cite{BCSS}).

\subsection{\it The condition $\alpha>\frac1q-\frac1p$\, if\, $r>2$}  This follows from

\begin{proposition}\label{kakeyaprop}
Let $p,q,r\in [2,\infty)$. Let $\eta$ be a radial $C_c^\infty$ function
satisfying $\eta(\xi)=1$ for $1/4\le |\xi|\le 12$.
Define $\fa_\la$ by
\Be\label{fapqr}
\fa_\la(p;q,r)=\sup_{\|f\|_p\le 1}
\Big\|\Big(\int_{1/2}^1 |e^{it\Delta} \eta(\tfrac D\la ) f|^r dt\Big)^{1/r}
\Big\|_{L^q(\bbR^d)}.
\Ee
Then for  $\la\gg 1$,
\Be\label{fapqrlowerbd}
 \fa_\la(p;q,r)\ge c \la^{1/q-1/p}(\log \la)^{1/2-1/r}.
\Ee
\end{proposition}

\begin{proof}
In what follows we set
$$\cA_4(\la^2)= \{x:3\la^2\le |\xi|\le 4\la^2\}.$$
By Lemma \ref{Schrimpliesrestr}, with parameters $a_0=3$, $a_1=4$,
for $\la\gg 1$
\begin{equation*}
\sup_{\|f\|_{L^p(\bbR^d)}\le 1}\Big(\int_{\cA_4(\la^2)}\Big(\int_{\la^2}^{2\la^2}
\big|\cE f \big(\tfrac{ s  }{\la^2} \xi,  s  \big)\big|^r
d s  \Big)^{\frac qr}d\xi\Big)^{\frac1q}\lesssim \fa_\la (p;q,r) \la^{- d+\frac
dp +\frac dq+\frac 2r}\,.
\end{equation*}
Let $$Tf(\xi, s  ) =
\cE f(\tfrac{ s  }{\la^2} \xi,  s  ).$$
Using
Khintchine's inequality we also get
\Be
\label{lqlrext3}
\sup_{\|\{ f_j\}\|_{L^p(\ell^2)} \le 1 }
\Big(\int_{\cA_4(\la^2)}\Big(\int_{\la^2}^{2\la ^2}
\big(\sum_j|Tf_j|^2\Big)^{\frac r2}
 d s  \Big)^{\frac qr}d\xi\Big)^{\frac1q}\lc \fa_\la (p;q,r) \la^{- d+\frac dp +\frac dq+\frac 2r}.
\Ee

For integers  $|j|\le \la/10$, let $z^j= (\la^{-1}j,0,\dots, 0)$ in
$\bbR^d$. Let $I_j=\{y: |y-z^j|\le (100 d\la)^{-1}\}.$ Let
$$R_j=\{(\xi, s  )\in \bbR^{d+1}: |\xi_1- 2j\la^{-1} s  |\le 10^{-1}\la,\, |\xi_i|\le
10^{-1}\la, \quad i=2,\dots, d,\,| s  |\le 100^{-1}\la^2\}.$$ For a
pointwise lower bound we use the following  lemma.

\begin{lemma}\label{lower} Let $a\in \bbR^d$, $b\in \bbR$,
and $g_j(y)= \chi_{I_j}(y) e^{i\inn a y-ib|y|^2}$. Then there is a
constant $c>0$, independent of $\la$, $j$ so that
\begin{equation*}
\Re\big[e^{i\inn{\xi-a}{z^j}-i( s  -b)|z^j|^2}
\cE[g_j](\xi, s  )\big]
 \ge c\la^{-d}, \text{ if } (\xi, s  )\in R_j+(a,b)\,.
\end{equation*}
\end{lemma}

\begin{proof} Let $I_0=\{y: |y|\le (100 d\la)^{-1}\}.$
We have
\begin{align*}
\cE g_j(\xi, s  )&= \int e^{i  s  |y|^2-i \inn{\xi}{y}} g_j(y) \, dy
=
\int e^{-i\inn {\xi-a}{z^j+h}+i( s  -b)|z^j+h|^2} \chi_{I_j}(z^j+h) dh
\\
&=e^{-i\inn{\xi-a}{z^j}}e^{i( s  -b)|z^j|^2}
\int e^{-i(\inn {\xi-a-2( s  -b)z^j}{h})} e^{i( s  -b)|h|^2}\chi_{I_0}(h) dh
\end{align*}
The pointwise lower bound follows quickly.
\end{proof}

Let $N_\la $ to be the largest integer which is smaller than
$\la/10$. By making use of the Besicovich set construction of
Keich  \cite{keich},  there are vectors
$v_j\in \bbR^{d+1}$ such that $v_j=a_je_1+b_je_{d+1}$ for some $a_j,
b_j\in \bbR$,
$v_j+R_j\subset\{(\xi, s ): \lambda^2\le  s  \le 2\la^2\}$, and   
$$
\meas\Big(\bigcup_{j=1}^{N_\la }(v_j+R_j)\Big) \lc
\frac{\la^{d+3}}{\log\la}.
$$
This  is just an obvious extension of
the two dimensional construction which gives a collection of rectangles
$\{R_j^{[2]}\}$ and vectors $(a_j,b_j) $ such that
$\meas\big(\bigcup_{j=1}^{N_\la }(v_j+ R_j^{[2]})\big) \lc
\frac{\la^{4}}{\log\la} $ and $(a_j,b_j)+ R_j^{[2]}\subset
\{(\xi_1, s ): \lambda^2\le  s  \le 2\la^2\}$.

 Let $\Phi(\xi, s )= (\tfrac{ s }{\la^2}
\xi, s )$ which is $1$--$1$ on $\cA_4(\la^2)\times [\la^2,  2\la^2]$,
and has Jacobian $J_\Phi$ with $|\det( J_\Phi(\xi, s ))|\sim 1$. Let $$
V_j:= \Phi^{-1}(v_j+R_j) \cap \big(\cA_4(\la^2)\times[\la^2, 2\la^2]\big),
\qquad E:= \bigcup_{j=1,\dots, N_\la } V_j.
$$
Then it follows that
\Be\label{measE}\la^{d+2} \lc \meas(V_j), \qquad
\meas(E)\lc\frac{\la^{d+3}}{\log\la}\,.
\Ee
Let $f_j(y) = \chi_{I_j}(y)e^{i\inn{a_j}{y}-ib_j|y|^2}$. Then by  Lemma \ref{lower},
\Be \label{ptwlowerbdT}
|T f_j(\xi)|\gc \la^{-d}, \qquad \xi\in V_j,
\Ee
and
\Be\label{Lpl2est}
\Big\|\Big(\sum|f_j|^2\Big)^{1/2}\Big\|_p\lc \la^{(1-d)/p}.
\Ee

We now modify arguments in  \cite{BCSS}.
By  \eqref{measE}  and \eqref{ptwlowerbdT},  we have
\begin{align}\label{summeasVjone}
\la^{d+3}&\lc N_\la \la^{d+2}\lc
\sum_{j=1}^{N_\la }\meas (V_j) \\ &=
\int_E \sum_{j=1}^{N_\la } \chi_{V_j}(\xi, s )\, d s \,d\xi \lc \la^{2d}\int_E \sum_{j=1}^{N_\la }
 |T f_j(\xi, s )|^2 d s \,d\xi,\nonumber
\end{align}
and by applications of H\"older's inequality,
\Be\label{summeasVjtwo}
\la^{2d}\int_E \sum_{j=1}^{N_\la }
 |T f_j(\xi, s )|^2 d s  d\xi\lesssim \la^{2d}  A\cdot B,
\Ee
where
\begin{align*}
A&=\Big(\int_{\cA_4(\la^2)}\Big(\int_{\la^2}^{2\la^2}\Big(\sum_j|Tf_j(\xi, s )|^2\Big)^{\frac r2} d s  \Big)^{\frac qr}d\xi\Big)^{\frac 2q},\\
B&=\Big(\int_{\cA_4(\la^2)}\Big(\int_{\la^2}^{2\la^2}  \chi_{E}(\xi, s ) \,d s
\Big)^{\frac{(q/2)'}{(r/2)'}} d\xi \Big)^{1-\frac 2q}.
\end{align*}
From \eqref{lqlrext3} and \eqref{Lpl2est}  we obtain,
\Be\label{aa}
A\lesssim \Big( \la^{\frac{1-d}{p}}\fa_\la (p;q,r) \la^{- d+\frac 1p +\frac dq+\frac 2r} \Big )^2.
\Ee

In order to  estimate $B$  we set
$$\fv(\xi)=\int_{\la^2}^{2\la^2} \chi_E(\xi, s ) \, d s ,$$
 the measure of the vertical
cross section of $E$ at $\xi$.
For $M>0$, we break
$$
B\lesssim
\Big(\int_{\{\xi\in \cA_4(\la^2)\,:\, \fv(\xi)\le M\}}
\fv(\xi)^{\frac{(q/2)'}{(r/2)'} }
d\xi\Big)^{1-\frac 2q} +
\Big(\int_{\{\xi\in \cA_4(\la^2)\,:\, \fv(\xi)> M\}}
\fv(\xi)^{\frac{(q/2)'}{(r/2)'} }
d\xi\Big)^{1-\frac 2q}\,.
$$
From the construction  of $E$ it is obvious that  $\fv$ is supported  in a tube where
$|\xi_1|\le C\la^2$ and $|\xi_i|\le C\la$, $2\le i\le d$, so
that
$$\Big(\int_{\{\xi\in \cA_4(\la^2)\,:\, \fv(\xi)\le M\}}
\fv(\xi)^{\frac{(q/2)'}{(r/2)'} }
d\xi\Big)^{1-\frac 2q} \lc
M^{1-\frac 2r} \la^{(d+1)(1-\frac 2q)}.
$$
Moreover since  $r\le q$ and therefore $(1-\tfrac{(q/2)'}{(r/2)'})\ge 0$,   by \eqref{measE}
\begin{align*}
\Big(\int_{\{\xi\in \cA_4(\la^2): \fv(\xi)>M\}}
\fv(\xi)^{\frac{(q/2)'}{(r/2)'} }
d\xi\Big)^{1-\frac 2q} &\lc
\Big(\int \fv(\xi)
M^{\frac{(q/2)'}{(r/2)'}-1 }
d\xi\Big)^{1-\frac 2q}
\\
&\le M^{\frac 2q-\frac 2r} \meas(E)^{1-\frac 2q}
\lc
M^{\frac 2q-\frac 2r} \Big(\frac{\la^{d+3}}{\log\la}\Big)^{1-\frac 2q}\, .
\end{align*}
Combining these two bounds, we have
\[
B\lc M^{-2/r} \la^{(d+3)(1-\frac 2q)} \big[M \la^{-2(1-\frac 2q)}+
M^{\frac 2q} (\log \la)^{\frac 2q-1}\big],
\] and choosing $M= \la^2 (\log\la)^{-1}$, which optimizes the above,  we obtain
\Be \label{besicovclaim}
B\lc \la^{(d+3)(1-\frac2q)}
\la^{\frac 4q-\frac 4r} (\log \la)^{\frac 2r-1}\,.
\Ee

Finally, we combine \eqref{besicovclaim}, \eqref{aa}, \eqref{summeasVjtwo} and \eqref{summeasVjone} to obtain
$$\la^{d+3}
\lc  \la^{2d} \la^{(d+3)(1-\frac2q)}
\la^{\frac 4q-\frac 4r} (\log \la)^{\frac 2r-1}\,
 \big[ \la^{\frac {1-d}{p}}
\fa_{\la}(p;q,r)
\la^{-d+\frac{d}{p}+\frac {d}{q}+\frac {2}{r}}\big]^2 \,,
$$
which yields  $\fa_\la(p;q,r)\ge c(\log \la)^{\frac 12-\frac 1r}
\la^{\frac 1q-\frac 1p}$.
\end{proof}

\subsection{\it Relation with Bochner--Riesz and the condition $\alpha>  0$\  if\,  $r=2$, $p=q>2$, $d\ge 2$. }\label{secondbesch}
The $L^p\to L^p(L^2(I))$  estimate implies sharp
results for the Bochner--Riesz multiplier in the same way as the wave equation (\cf. \S7 in \cite{mss}).

For small $\delta>0$, let us set
$h_\delta(\xi)= \phi(\delta^{-1}(1-|\xi|^2))$ with $\phi\in C_c^\infty(-1,1)$. Let $\psi$ be radial, supported in $\{1/2<|\xi|<2\}$ so that $\psi=1$ on the support of $h_{\delta}$.
Then by the Fourier inversion formula and the support property of $\psi$ it follows that
$$
h_\delta (D)f= \frac {1}{2\pi}  \int_{-\infty}^\infty
\delta\widehat \phi(\delta s )\, e^{i s } e^{i s \Delta} \psi(D) f\, d s .
$$
By the Schwarz inequality  we get
$$|h_\delta (D)f|\le
\Big(\int |\delta\widehat\phi(\delta  s )|d s \Big)^{1/2}
\Big(\int|e^{i s \Delta }\psi(D) f|^2 |\delta\widehat \phi (\delta s )|d s \Big)^{1/2}.
$$
Thus we  see that
$$
\|h_\delta\|_{M_p} \lc   \sup_{\|f\|_p\le 1}\Big \| \Big(\int|e^{i s \Delta }\psi( D) f|^2 |\delta\widehat \phi (\delta s )|d s \Big)^{1/2} \Big \|_p,
$$
which after rescaling becomes
$$
\|h_\delta\|_{M_p} \lc
\sup_{\|f\|_p\le 1}\Big \| \Big(\int|e^{it\Delta }\psi( \sqrt\delta D) f|^2 |\widehat \phi (t)|dt\Big)^{1/2} \Big \|_p .$$
Hence, using the rapid decay of $\widehat \phi$ and a further rescaling we see that the  sharp bound $\|h_\delta\|_{M_p} \lc\delta^{1/2-d(1/2-1/p)}$, for $p>2+\frac{2}{d-1}$, would follow from  $U: B^p_{\alpha,\nu}\to L^p (L^2(I))$, with $\alpha=d(1-\frac2p)-1$, for any $\nu>0$.

 We see that the
$L^p\to L^p (L^2(I))$ inequality for some $p> 2$  would imply that
$h_\delta$ is a multiplier of $\cF L^p$ with bounds independent of $\delta$.
However   a variant of Fefferman's argument for the
ball multiplier \cite{F}, based on a Kakeya set argument,
 shows that
 \Be \label{hbound2}
 \|h_\delta\|_{M_p} \gc \log (1/\delta)^{1/2-1/p}.
 \Ee
This
 establishes the final necessary condition $(vi)$
in Proposition \ref{necprop}.
For completeness we  include some details of the argument.

\begin{proof}[Proof of \eqref{hbound2}]
By de Leeuw's theorem it suffices to prove the lower bound for $d=2$.
We may assume that $\delta<10^{-10}$. By Khintchine's inequality, we have
\begin{equation}\label{MZineq}\Big\|\big(\sum_\nu \big| h_\delta(D)  f_\nu\big|^2\big)^{1/2}
\Big\|_p \lc \|h_\delta\|_{M_p}
\Big\|\big(\sum_\nu |f_\nu|^2\big)^{1/2}\Big\|_p.
\end{equation}
For $\nu\in \mathbb Z\cap [-10^{-2}\delta^{-1/2}, \, 10^{-2}\delta^{-1/2}]$, let us set
\[h_{\delta,\nu}(\xi)= h_\delta(\xi) \phi(\delta^{-1/2}\xi_1-\nu)
\chi_{+}(\xi), \quad \xi=(\xi_1,\xi_2)\in \mathbb R^2\]
where $\chi_+ $ is the characteristic function of the upper half plane.
Define $T_\nu$ by $\widehat{T_\nu f}= h_{\delta,\nu}\widehat f$.
Let $\eta_\nu$ be the inverse Fourier transform of a  bump function which is
 supported on a ball of radius $C\delta^{-1/2}$ so that $\eta_\nu(\xi)=1$ for $\xi$ in  the support of
 $h_{\delta,\nu}$. Define  $\Phi_\nu$ by
$\widehat\Phi_\nu (\xi)=\eta_\nu(\xi)\phi(\delta^{-1/2}\xi_1-\nu)\chi_{+}(\xi)$.
Then
$|\Phi_\nu(x)|\lc\delta^{-d/2}(1+\delta^{-1/2}|x|)^{-{(d+1)}}$ for the $\nu$'s under consideration, so that
$\|\{\Phi_\nu*g_\nu\}\|_{L^p(\ell^2)}\lc
\|\{g_\nu\}\|_{L^p(\ell^2)}$.
Since $T_\nu g= h_\delta(D)[\Phi_\nu*g]$,
inequality \eqref{MZineq} applied to $f_\nu=\Phi_\nu*g_\nu$
implies that
\Be\label{mz}\Big\|\big(\sum_\nu |T_\nu g_\nu|^2\big)^{1/2}
\Big\|_p \lc \|h_\delta\|_{M_p}
\Big\|\big(\sum_\nu |g_\nu|^2\big)^{1/2}\Big\|_p.
\Ee

Let $\theta_\nu= (\delta^{1/2}\nu, \sqrt{1-\delta\nu^2})$, let $\theta_\nu^\perp$ be a unit vector perpendicular to $\theta_\nu$ and
\begin{align*} R_\nu&= \big\{(x_1,x_2): \,
|\inn{x}{\theta_\nu}|\le 10^{-2}\delta^{-1},\,
|\inn{x}{\theta_\nu^\perp}|\le 10^{-1}\delta^{-1/2}\big\}\,.
\end{align*}
Letting $f_\nu(y)=\chi_{R_\nu}(y)e^{i\inn{\theta_\nu}{y}}$, we have that
\Be \label{ptwlowerbound}
|e^{-i\inn {x}{\theta_\nu}}T_\nu f_\nu (x)|\ge c>0 \text{ for } x\in R_\nu\,.
\Ee

Here we use again  the sharp bounds in the construction
of   Besicovich sets  \cite{keich}.
There are vectors $a_\nu$, $|\nu|\le 10^{-2}\delta^{-1/2} $ so that
with
$E:=  \bigcup_\nu R_\nu$ the measure of $E$ is
$O(\delta^{-2}/\log\delta^{-1})$
but
the corresponding translations $a_\nu+ R^\nu$ have $O(1)$ overlap.
Define $g_\nu(x)= f_\nu(x-a_\nu)$, which is supported in $a_\nu+ R_\nu$. Then $|T_\nu g_\nu|\ge c$ on
$a_\nu+R_\nu$.
Thus we get
$$
\delta^{-2} \lc \sum_\nu|R_\nu| \lc \sum_\nu\int \chi_{a_\nu+R_\nu}(x)\, dx
\lc \int_E \sum_\nu|T_\nu g_\nu|^2dx
$$ and also by H\"older's inequality  and \eqref{mz}
the last one in the above string of  inequalities is bounded by
\begin{align*}
\meas(E)^{1-2/p} \Big\|\Big(\sum_\nu|T_\nu g_\nu|^2\Big)^{1/2}\Big\|_p^2
&\lc \|h_\delta\|_{M_p} ^2\Big(\frac{\delta^{-2}}{\log\delta^{-1}}\Big)^{1-2/p}
\Big\|\Big(\sum_\nu| g_\nu|^2\Big)^{1/2}\Big\|_p^2\,.
\end{align*}
Now by  the bounded overlap of the  translated  rectangles
$a_\nu+ R_\nu$, we see
\begin{align*}
\Big\|\Big(\sum_\nu| g_\nu|^2\Big)^{1/2}\Big\|_p^2
\lc \Big(\int\sum_\nu\chi_{a_\nu+R_\nu}dx\Big)^{2/p}
\lc \Big(\sum_\nu|R_\nu|\Big)^{2/p} \lc\delta^{-4/p}.
\end{align*}
Combining the three displayed inequalities we get
$\delta^{-2}\lc \|h_\delta\|_{M_p} ^2 (\delta^{-2}/\log\delta^{-1})^{1-2/p}\delta^{-4/p}$
and thus the desired \eqref{hbound2}.
\end{proof}

\section{Proof of Theorem \ref{onedimp-prec}}\label{d=1}
The lower bounds for $\fA_\la(p;q,r)$ were established in the previous section, and here we prove the upper bounds, mainly by interpolation arguments. By Lemma \ref{smallerintervals}, we can take $I=[1/2,1]$.

\subsection{\it Proof of $(i)$} We consider the cases $\frac1q+\frac1r\ge \frac12$ and $\frac1q+\frac1r< \frac12$ separately.

\subsubsection*{The case $\frac1q+\frac1r\ge \frac12$}
Note that the set
\[\big\{\,(\tfrac1p, \tfrac1q,\tfrac1r)\,:\,  2\le r\le p\le q\le \infty,\quad \tfrac1q+\tfrac1r\ge \tfrac12\,\big\}\]
is the closed tetrahedron with vertices $(\frac14, \frac14, \frac14)$, $(\frac12,\frac12,\frac12)$,
$(\frac12, 0,\frac12)$, and $(0, 0,\frac12)$.  Hence by interpolation  it is enough to show the estimate
\Be\label{pqr1}
\fA_\la(p;q ,r)\lc \la^{\frac1q-\frac1p} [\log \la]^{\frac12-\frac1r}
\Ee
for $ (p, q,r) =(4, 4, 4)$, $(2,2,2)$,
$(2, \infty , 2)$ and $(\infty, \infty, 2)$.   The estimate for $(p,q,r)=(2,2 ,2)$ is immediate from Plancherel's theorem.  More generally we recall from  \cite{lerose} the estimate $\fA_\la(p;p,2)\lc 1$ with
$2\le p\le \infty$,
 which is related to a square-function estimate for
 equally spaced intervals. So we also get the estimates for  $(p,q,r)=(\infty,\infty ,2)$.
For $(2,\infty ,2)$ we choose a nonnegative $\chi_{\rm{o}}\in C_c^\infty(\bbR)$, so that $\chi_{\rm{o}}(t)=1$ on $[1/2,1]$.
We need to estimate, for fixed $x$,
 \begin{equation*}\int \chi_{\rm{o}}(t) |U\eta(\tfrac D\la ) f(x,t)|^2 dt =
\frac{1}{(2\pi)^{2d}}\iint e^{ix(\xi-w)} \widehat
f(\xi)\overline{\widehat f(w) }
\eta(\tfrac{\xi}{\la})\overline{\eta(\tfrac{w}{\la}) }
\widehat{\chi_{\rm{o}}}(|\xi|^2-|w|^2)\, d\xi\,dw
\end{equation*} and since $|\xi|+|w|\ge \lambda$,  the above is bounded by
$$C_N \int_{\R}\int_{\R}  \big(1+\la\big||\xi|-|w |\big|\big)^{-N} |\widehat f(\xi)||\widehat f(w) | \,d\xi\,dw \,\lc\, \la^{-1}\|\widehat f\|_2^2.$$
This gives the desired estimate for $(p,q,r)=(2,\infty ,2)$. For   $(p,q,r)=(4,4,4)$ we use the bound
\[\Big(\iint \Big|\psi(\xi, s )\int_{|y|\le 1} f(y)\, e^{i \lambda( s  |y|^2-\xi y)} f(y)
   \,dy\Big|^4d\xi d s \Big)^{1/4}\lc   \lambda^{-\frac12}(\log\lambda)^\frac14 \|f\|_4,\]
   where $\psi\in C_c^\infty$. This is implicit in \cite{hoer2}
(see also \cite{ms} for more discussion and related issues).
Then by rescaling,  Lemma \ref{resimpliesSchr} and Lemma \ref{smallerintervals}
we get  \eqref{pqr1} for $(p,q,r)=(4,4,4)$.

\subsubsection*{The case   $\frac1q+\frac1r<\frac12$}
We begin as before by
observing  that the set
\[\Delta_1=\big\{\,(\tfrac1p, \tfrac1q,\tfrac1r)\,:\,  2\le r\le p\le q\le \infty,\quad \tfrac1q+\tfrac1r< \tfrac12\,\big\}\]
is the closed tetrahedron with vertices $(0,0,0),$ $(\frac14, \frac14, \frac14)$
$(\frac12, 0,\frac12)$ and $(0, 0,\frac12)$, from which the triangle with vertices $(\frac14, \frac14, \frac14)$
$(\frac12, 0,\frac12)$, and $(0, 0,\frac12)$ is removed.
We use a bilinear analogue of our adjoint restriction operator, and
rely on rather elementary estimates from \cite{hoer2}.
Define $\chi_\ell$ so that $\sum_{\ell\in \bbZ} \chi_\ell\equiv 1$, $\chi_\ell = \chi_1(2^\ell\cdot)$ and
$\chi_1$ is supported in $(2, 8)$. Let
 $$\fB_{\la,\ell}[f,g]\,=\,
\iint_{[-1,1]^2}
e^{i s (|y|^2+|z|^2)-i\frac{ s }{\la^2}\xi({y+z})}\chi_\ell(|y-z|) f(y)g(z) \,
dy dz,
$$
so that
$$
(\cE f \cE f)(\frac{ s }{\lambda^2}\xi,  s ) =\sum_{\ell\ge 0} \fB_{\la,\ell}
(f,f)(\xi, s ).
$$
We shall verify that for $\ell\ge 0$
\Be 
\label{pqrbilinear}
\|\fB_{\la,\ell} (f,g)\|_{L^{q/2}(\cA(\la^2); L^{r/2}[\la^2, 2\la^2])}
 \lc 2^{-2\ell(\frac12-\frac 1q-\frac 1r)}
\|f\|_p \|g\|_p
\Ee 
when  $(\frac1p,\frac1q,\frac1r)$ is contained in the closed tetrahedron with vertices $(0,0,0),$ $(\frac14, \frac14, \frac14)$
$(\frac12, 0,\frac12)$ and $(0, 0,\frac12)$.  By summing a geometric series, this yields \eqref{adjrestrassu}  for $(\frac1p,\frac1q,\frac1r)\in \Delta_1$, which by  Lemmata \ref{resimpliesSchr} and \ref{smallerintervals}, yields the desired
\Be
\label{pqr2}
 \fA_\la(p;q,r)\lc \la^{1-\frac 1p-\frac 1q-\frac 2r}.
 \Ee
We remark that conversely, if \eqref{pqr2} holds,
then
we can use
 Lemma \ref{Schrimpliesrestr} and a Fourier
expansion of $\chi_\ell(y-z)$ to bound
the left hand side of \eqref{pqrbilinear}  by $C\|f\|_p\|g\|_p$,
with $C$ independent of $\ell$.

It remains to show \eqref{pqrbilinear}. By interpolation it is enough to do this with $(p,q,r)=(\infty,\infty,\infty),$ $(4, 4, 4)$
$(2, \infty,2)$, and $(\infty, \infty,2)$. The last two estimates were already obtained;
 note that the bounds \eqref{pqr1} and \eqref{pqr2} coincide
for the cases   $(p,q,r)=(2, \infty,2)$ and $(\infty, \infty,2)$ and  the bounds
for \eqref{pqrbilinear} are independent of $\ell$. Hence  from the  bounds \eqref{pqr1}  previously obtained
and the discussion above  we have the required bounds for
$(p,q,r)=(2, \infty,2)$, and $(\infty, \infty,2)$.
 We  note  that the argument for the proof of the endpoint adjoint restriction  theorem in \cite{hoer2}
gives
\Be \label{44bil}
\|B_{\la,\ell} (f,g)\|_{L^2_{\xi, s }} \lc \|f\|_4 \|g\|_4,
\Ee
uniformly
 in $\ell\ge 0$, where $B_{\la,\ell} (f,g) (\xi, s )=\fB(f,g) (\tfrac{\la^2}{ s }\xi, s )$, and by a change of variables we obtain
\eqref{pqrbilinear}
holds with $(p,q,r)=(4,4,4)$.
To get the inequality \eqref{pqrbilinear}  for $(p,q,r)=(\infty,\infty,\infty)$
we  need to  integrate $\chi_\ell(|y-z|)$ over
$[-1,1]^2$ which yields  the gain of $2^{-\ell}$.

\subsection{\it Proof of $(ii)$} We also consider the cases $1-\frac 1p>
\frac 2q+\frac 1r$ and $1-\frac 1p\le
\frac 2q+\frac 1r$ separately.

\subsubsection*{The case $1-\frac 1p\le
\frac 2q+\frac 1r$} We note that the set
\[ \Delta_2=\big\{\,(\tfrac1p, \tfrac1q,\tfrac1r)\,:\,  2\le p<r\le q\le \infty,\quad \tfrac 2q+\tfrac 1r \ge 1-\tfrac 1p\,\big\}\]
is the closed tetrahedron with vertices $(\frac12,\frac12,\frac12),$
$(\frac14, \frac14, \frac14)$ $(\frac12, \frac16,\frac16)$ and
$(\frac12, 0,\frac12)$, from which the face with vertices
$(\frac12,\frac12,\frac12),$ $(\frac14, \frac14, \frac14)$ and
$(\frac12, 0,\frac12)$ is removed. Note that from the previous
bounds \eqref{pqr1} and \eqref{pqr2}  we already have the required
bounds \Be\label{pqr3}\fA_\la(p;q,r)\lc \lambda^{\frac1q-\frac1r}\Ee
for $(p,q,r)=(2,2,2)$ and $(2,\infty,2)$. Obviously $\Delta_2$ is
contained in  the convex hull of $(\frac12, 0,\frac12)$,
$(\frac12,\frac12,\frac12),$ and the half open line segment
$[(\frac12, \frac16,\frac16), (\frac14, \frac14, \frac14))$. Hence
by it is enough to show \eqref{pqr3} for $(\frac1p,\frac1q,\frac1r)$
contained in the half closed line segment $[(\frac12,
\frac16,\frac16), (\frac14, \frac14, \frac14))$. But these follow
from Lemmata \ref{resimpliesSchr} and \ref{smallerintervals}, combined with the
restriction estimate for the parabola which gives
\eqref{adjrestrassu} for $(\frac1p,\frac1q,\frac1r)\in [(\frac12,
\frac16,\frac16), (\frac14, \frac14, \frac14))$.

\subsubsection*{The case  $1-\frac 1p>
\frac 2q+\frac 1r$}
We note that the set
\[ \big\{\,(\tfrac1p, \tfrac1q,\tfrac1r)\,:\,  2\le p<r\le q\le \infty,\quad \tfrac 2q+\tfrac 1r < 1-\tfrac 1p\,\big\}\]
is contained in the  quadrangular pyramid $\mathcal Q$ with vertices $(0, 0, 0)$,  $(\frac12,0,0)$, $(\frac14, \frac14, \frac14)$
$(\frac12, \frac16,\frac16)$, and $(\frac12, 0,\frac12)$.  We need to show \eqref{pqr2} for $(\frac1p,\frac1q,\frac1r)$ contained in the above set. Repeating the above argument,  the asserted estimates follow if we establish,  for $\ell\ge 0$ and $(\frac1p,\frac1q,\frac1r)\in \mathcal Q$,
\Be
\label{pqrbilsecond}
\|\fB_{\la,\ell} (f,g)\|_{L^{q/2}(\cA(\la^2); L^{r/2}[\la^2, 2\la^2])}
 \lc 2^{-\ell(1-\frac 1p -\frac 2q-\frac 1r)}
\|f\|_p \|g\|_p.
\Ee
We only need to verify it for
 $(p,q,r)=(\infty, \infty, \infty)$,  $(4, 4, 4)$,  $(2, \infty,2)$,
$(2, 6,6)$, and $(2,\infty,\infty)$. The first three cases were
already obtained when we showed \eqref{pqrbilinear}, and the case
$(p,q,r)=(2, 6,6)$ follows from the linear adjoint restriction
estimate for the parabola as before. Finally the case
$(p,q,r)=(2,\infty,\infty)$ with a gain of $2^{-\ell/2}$ follows
from the Schwarz inequality, and so we are done.


\section{Sharper regularity results}\label{fifth}

\subsection{\it Combining frequency localized  pieces}\label{cflp}
One can use the uniform regularity results for the frequency localized pieces
to prove sharper bounds such as  Sobolev estimates by using arguments
based on the Fefferman--Stein \#-function.
Let $\varphi$ be a radial  smooth function supported in
$\{\xi:1/4<|\xi|<4\}$, not identically $0$. Let $I=[-1,1]$ and  \Be\label{Gammadef}
\Gamma(p;q,r)=\sup_{\la>1} \la^{-d(-\frac 1p-\frac 1q)+\frac 2r} \big\|U
\varphi \big(\tfrac{D}{\la}\big)\big\|_{L^p\to L^q(\mathbb R^d;L^r(I))}
\Ee
It is not hard to verify that  the finiteness of
$\Gamma(p;q,r)$ is independent of the particular choice of $\varphi$.
The following statement is a special case of the result in the appendix of \cite{lerose}.
\begin{proposition}\label{combin} Let $p_0,q_0,r_0\in [1,\infty]$,
$q\in(q_0,\infty)$, $r_0\le r<\infty$, $p_0\le q_0$ and assume
$1/p_0-1/q_0=1/p-1/q$.
Suppose that
$\,\Gamma(p_0;q_0,r_0)<\infty.$
Then
\begin{equation}
\label{dyadSchrqr}
\Big\|\, \Big(\int_I|U\!f(\cdot,t) |^r dt\Big)^{1/r}
\Big\|_{L^q(\mathbb{R}^d)} \lc
\|f\|_{B^p_{\alpha,q}(\bbR^d)}\,,
\quad
\alpha=d\big(1-\tfrac 1p-\tfrac 1q\big)-\tfrac 2r\, .
\end{equation}
If $f\in B^p_{s,q}(\bbR^d)$ with $s=d(1-\tfrac 1p-\tfrac 1q)$, then for almost every $x\in \bbR^d$ the function
$t\mapsto U^a f(x,t)$ is  locally in $B^q_{1/q,\nu}(\bbR)$, and (thus)
continuous, and
\begin{equation*}
\big\|\sup_{t\in I} |U \!f(\cdot,t)|\,  \big\|_{L^q(\bbR^d)}
 \lc
\|f\|_{B^p_{s,q}(\bbR^d)}, \quad s=d\big(1-\tfrac 1p-\tfrac 1q\big)\,.
\end{equation*}
\end{proposition}

The  Sobolev estimates follow from this since for $q\ge p\ge 2$ one has
$L^p_\alpha\subset B^p_{\alpha,p} \subset B^p_{\alpha,q}$. We note that the result in \cite{lerose} is slightly sharper. Namely the left hand side of
\eqref{dyadSchrqr} can be replaced by the $L^q(\bbR^d)$ norm of
$(\sum_{k>0} (\int_I |P_k U\!f(\,\cdot\,,t) |^r dt)^{\nu/r})^{1/\nu}$, where $\nu>0$.

\begin{proof}[Proof of Corollaries
\ref{sobcor1} and \ref{SobCor2}]
Proposition \ref{combin} implies  the validity of  the corollaries
given their analogues for
frequency localized functions (namely Theorems \ref{onedimp-prec} and
\ref{besov-vs-ext}).  For Corollary \ref{SobCor2} we use that $\text{\rm R}^*(p_0\to q_0)$ implies $\text{\rm R}^*(p\to q_0)$ for all $p\ge p_0$.
\end{proof}

\subsection{\it A remark on recent results by
Bourgain and Guth}\label{bgpq}
As mentioned in the introduction, the recent results in  \cite{bogu} on
$\text{\rm R}^*(q\to q)$
give  results  on the sharp $L^q_\alpha\to L^q(\bbR^d\times I)$
boundedness of $U$. In a restricted range they also imply new results
on $R^*(p\to q)$ with the best possible $p=p(q)$ which Tao \cite{ta} had proved
for $q>\frac{2(d+3)}{d+1}$, and likewise one then obtains corresponding results for the Schr\"odinger operator.
The following statement  is proved by
a simple interpolation argument for bilinear operators.

\begin{proposition}\label{bg}
Suppose that  $\text{\rm R}^*(q_0\to q_0)$
 holds for some
$q_0\in (2,\frac{2(d+3)}{d+1})$.  Then

\noindent (i) $\text{\rm R}^*(p\to q)$ holds with
$q=\frac{d+2}{d} p'$
provided that
\begin{equation*}
q>q_*:=\frac{2(d+3)}{d+1}\Big(1-\gamma(d,q_0)\Big),\quad
\text{ where }\ \gamma(d,q_0)=
\frac{\frac{1}{q_0}-\frac{d+1}{2(d+3)}}
{\frac{d+1}{2d}-\frac{d+2}{dq_0}}.
\end{equation*}

\noindent (ii) Let $q_*<q<\infty$, $q\le r\le \infty$
and suppose that $0\le \frac{1}{p}-\frac{1}{q}<
1-\frac{2(d+1)}{dq_*}$.
 Then $U:
L^p_{\alpha}(\bbR^d)\to L^q(\bbR^{d};L^r(I))$ is bounded with
$\alpha= d\big(1-\frac 1p-\frac 1q\big)-\frac 2r$.
\end{proposition}
In two dimensions $\text{\rm R}^*(q\to q)$
was proven in \cite{bogu} for $q>56/17$ and the sharp
inequality
$\text{\rm R}^*(p\to q)$  for $q=2p'$ follows for $q>13/4$.

\begin{proof}[Proof of Proposition \ref{bg}]
By
Theorem \ref{besov-vs-ext} and Proposition \ref{combin}
 it suffices to prove the first part.

Let $E_1$ and $E_2$ be $1/2$-separated sets in the unit ball of $\bbR^d$
and define $\cE_i f=\cE[f\chi_{E_i}]$.
By Theorem 2.2 in \cite{tavave}, it suffices to prove the estimate
\Be\label{bilBGsharp}
\big\|\cE_1 f_1 \cE_2 f_2\big\|_{q/2}\lc \|f_1\|_p\|f_2\|_p\
\Ee
for $q>q_*$ and $p$ in a neighbourhood of $\frac{dq}{dq-d-2}$ (i.e. the $p$ which satisfies $q=\frac{d+2}{d} p'$).

By hypothesis and H\"older's inequality,  \eqref{bilBGsharp}  holds with $p\ge q=q_0$.
By Tao's theorem
 \eqref{bilBGsharp}  holds with $p\ge 2$ and $q/2> \frac{d+3}{d+1}$.
The theorem then follows by interpolation of bilinear operators.
Indeed, we determine $\theta\in (0,1)$ and $q_*\in (q_0,\frac{2(d+3)}{d+1})$ by
$$\frac{1-\theta}2+\frac\theta{q_0}= 1-\frac{d+2}{dq_*}, \qquad
(1-\theta)  \frac{d+1}{d+3}+\theta
\frac 2{q_0} = \frac {2}{q_{*}}.
$$
We compute $\theta =
\big(\frac{d+2}{dq_*}-\frac 12\big)/ \big(\frac12-\frac{1}{q_0}\big)$
and $\theta=\big(\frac{1}{q_*}-\frac{d+1}{2(d+3)}\big)/\big(\frac{1}{q_*}-\frac{d+1}{2(d+3)}\big),$ from which we obtain
$1/{q_*}= \big(\frac{d+1}{2(d+3)}-\frac {b}{2}\big)\big/
\big(1-\frac{d+2}{d}b\big)$  with
$b=\big(\frac{1}{q_0}-\frac{d+1}{2(d+3)}\big)\big/
\big(\frac 12-\frac 1{q_0}\big)$. A further computation shows that $q_* $ is equal to  $\frac{2(d+3)}{d+1}\big(1-\gamma(d,q_0)) $ as in the statement of the lemma.
\end{proof}

\section{Proof of Theorem \ref{pop1}}\label{165sect}
\noi {\bf Definition.} Fix $d\ge 1$, and let $p,q,r\in [2,\infty]$.
For $N> 1$, let
 $$\Lambda_{p,q,r}(N,\rho)\equiv \Lambda_{p,q,r}(N,\rho;d)
=\sup
\big\|Uf_1\, Uf_2\|_{L^{q/2}(\bbR^d, L^{r/2}[0,\rho])}
$$
where the supremum is taken over all pairs of function $(f_1,f_2)$
whose  Fourier transforms are supported in $1$-separated subsets of
$\{\xi: |\xi-Ne_1|\le 2d\}$, and
which satisfy
$\|f_1\|_p, \|f_2\|_p\le 1$.

We remark that the unit vector $e_1$ does not play a special role here. It could be replaced by any unit vector, by rotational invariance.

By considering two bump functions, it is easy to calculate that
\Be\label{lowertriv}
\Lambda_{p,q,r}(N,\rho) \gc
 N^{\frac 2q-\frac 2r},\quad 1\le p,q,r\le \infty,
\Ee
whenever $\rho>1$,
and significant for Theorem \ref{pop1} is  the  following two dimensional estimate,
\Be\label{lrv-result}
\sup_{\rho>1}\Lambda_{2,q,r}(N,\rho;2) \lc
 N^{\frac 2q-\frac 2r}, \quad q>16/5, \quad r\ge 4\,,
\Ee
which was proven in \cite{lerova}
(see also \cite{le1} and \cite{ro} for related previous results).
We will combine this with the following two lemmata.
\begin{lemma}\label{2sthenps}
Let $p_0\le p\le q\le r$ and $\eps_{\rm{o}}>0$. Then, for $N,\rho>1$, \Be\label{p0pest}
\Lambda_{p,q,r}(N,\rho) \lc N^{\eps_{\rm{o}}}
\rho^{2d(\frac{1}{p_0}-\frac
1p)} \Lambda_{p_0,q,r}(N,\rho)\,.
\Ee
\end{lemma}

\begin{lemma}\label{billemma}
Let $2\le p\le  q\le r\le \frac{2q}{q-2}$ and $\eps>0$.
Let $\psi\in C_c^\infty$ be supported in the annulus
$\{\xi\in \bbR^d: 1/2\le |\xi|\le 2\}$. Then, for  $\la> 1$,
\begin{multline}\label{bilinform}
\big\|U\psi\big(\tfrac D\la\big) f\big\|_{L^q(\bbR^d;L^r[0,1])}
\\
\lc \Big(\la^{\frac{4}{q}-2d(\frac1p-\frac1q)}+\sup_{1< N< \la}
N^{\frac4r-2d(\frac1p-\frac1q)+\eps} \Lambda_{p,q,r}(N,C
\la^2/N^2)\Big)^{1/2} \la^{-\frac 2r+d(\frac 1p-\frac 1q)}
\|f\|_p\,.
\end{multline}
\end{lemma}

Lemma \ref{2sthenps} relies on a localization argument such as in
\cite{le}
and Lemma \ref{billemma} relies on a by now standard scaling argument in
\cite{tavave} which reduces estimates for bilinear  operators with separation assumptions to estimates for linear operators.

We may combine  \eqref{p0pest},  with $p_0=2$, and
\eqref{bilinform} to obtain
\begin{corollary}\label{2qrcor}
Let  $2\le p\le q\le r\le \frac{2q}{q-2}$.  Suppose that \Be\label{2qrassumption}
\sup_{\rho> 1} \Lambda_{2,q,r}(N,\rho;d)\lc N^{\gamma},\quad \text{
for some\quad $\gamma<2d\big(1-\tfrac 1p-\tfrac 1q\big)-\tfrac 4r$} .\Ee Then if
$d(1-\frac1p-\frac1q)-\frac2q\ge 0$, then for all $\la>1$,
\Be\label{dyadfreqconclusion} \big\|
U\psi\big(\tfrac D\la\big)f\big\|_{L^q(\bbR^d;L^r[0,1])} \lc \la^{d(1-\frac 1p-\frac 1q)-\frac
2r} \|f\|_p\,. \Ee
\end{corollary}

\noindent Supposing this for the moment we give the
\begin{proof}[Proof of Theorem \ref{pop1}]
By  Proposition \ref{combin}
it suffices to prove, in two spatial dimensions, the estimate
\eqref{dyadfreqconclusion}
for $p=q>16/5$ and $r\ge 4$. Using \eqref{lrv-result}, we
put $\gamma=2/q-2/r$ and verify that
the condition   \eqref{2qrassumption} with $d=2$   in the
range $p=q>16/5$ and $r\ge 4$.
Thus \eqref{dyadfreqconclusion} holds in this range, and we are done.
\end{proof}

\begin{proof}[Proof of Lemma \ref{2sthenps}]
Let $\eta_1$, $\eta_2$ be smooth, supported in balls of diameter 1/2
which are contained in $\{\xi:|\xi-Ne_1|\le 2d\}$, and which are separated by
1/2.  Define the  operators $S_1,\,S_2$ by
$\widehat{S_i f}(\xi,t)=\eta_i(\xi)\, \widehat
{Uf}(\xi)$, $i=1,2$.
It suffices to prove that
$\|S_1f_1\,S_2
f_2\|_{L^{q/2}(\mathbb{R}^d\!,\,L^{r/2}[0,\rho])}$ is dominated by $\|f_1\|_p\|f_2\|_p$ times a constant multiple of the expression on the right hand side of
\eqref{p0pest}.

We partition $\mathbb{R}^d$ into cubes $\mathcal{Q}_\nuv$ of side
 $\rho$ with centre $\rho\nuv\in \rho\mathbb{Z}^d$,
 and define \Be\mathcal{P}_\nuv=
\{ (x,t)\in \mathbb{R}^d\times [0,\rho]\,:\,x-2tNe_1\in
\mathcal{Q}_\nuv\}.
\Ee The parallelipipeds form a partition of
$\mathbb{R}^d\times[0,\rho]$.
For fixed $x$ the intervals
 $I^x_\nuv= \{t:(x,t)\in \cP_\nuv\}$ are disjoint.
Thus
\begin{align*}
\|F\|^{q/2}_{L^{q/2}(\mathbb{R}^d;\,L^{r/2}[0,\rho])}\le
\int_{\mathbb{R}^d}\Big(\sum_{\nuv}\int_{I^x_\nuv}|F(x,t)|^{r/2}dt\Big)^{q/r}
dx\le\sum_{\nuv}\big\|\chi_{\cP_\nuv}F\big\|_{L^{q/2}(\mathbb
R^d;L^{r/2}[0,\rho])}^{q/2};
\end{align*}
here
we used the triangle inequality for
$\|\cdot\|_{\ell^{q/r}}^{q/r}$ as $q/r\le 1$.

Taking $F= S_1f_1S_2f_2$, and
denoting by
$\mathcal{Q}_\nuv^*$, the enlarged cube with side
$50d\rho N^\e$, where $0<\eps<4d\eps_{\rm{o}}$, we obtain
\begin{align*}
\big\|
S_1f_1\,S_2f_2
\big\|^{q/2}_{L^{q/2}(\mathbb{R}^d;\,L^{r/2}[0,\rho])}&
\le\,
\sum_{\nuv} \|\chi_{\cP_\nuv}S_1f_1\,S_2f_2\|_{L^{q/2}(\mathbb
R^d;L^{r/2}[0,\rho])}^{q/2}\\&\lc
\sum_{\nuv}(I_\nuv^{q/2}+II_\nuv^{q/2}+III_\nuv^{q/2}+IV_\nuv^{q/2}),
\end{align*}
where
\begin{align}\label{fs}\begin{split}
I_{\nuv}&=
\big\|\chi_{\cP_\nuv}S_1[f_1\chi_{\mathcal{Q}_\nuv^*}]\,S_2[f_2\chi_{\mathcal{Q}_\nuv^*}]
\big\|_{L^{q/2}(\mathbb R^d;L^{r/2}[0,\rho])},
\\
 I\!I_\nuv&=
\big\|\chi_{\cP_\nuv}
S_1[f_1\chi_{\mathbb{R}^d\backslash\mathcal{Q}_\nuv^*}]\,S_2[f_2\chi_{\mathcal{Q}_\nuv^*}]
\big\|_{L^{q/2}(\mathbb R^d;L^{r/2}[0,\rho])},
\\
I\!I\!I_\nuv&=
\big\|\chi_{\cP_\nuv}S_1[f_1\chi_{\mathcal{Q}_\nuv^*}]\,S_2[f_2\chi_{\mathbb{R}^d\backslash\mathcal{Q}_\nuv^*}]
\big\|_{L^{q/2}(\mathbb R^d;L^{r/2}[0,\rho])},
\\
IV_\nuv&= \big\|\chi_{\cP_\nuv}
S_1[f_1\chi_{\mathbb{R}^d\backslash\mathcal{Q}_\nuv^*}]\,S_2[f_2\chi_{\mathbb{R}^d\backslash\mathcal{Q}_\nuv^*}]
\big\|_{L^{q/2}(\mathbb R^d;L^{r/2}[0,\rho])}.
\end{split}
\end{align}

First we consider the main terms $I_\nu$. By
H\"older's inequality,
\begin{align*}
I_\nuv &\le \Lambda_{p_0,q,r}(N,\rho)
\prod_{i=1}^2\|f_i\chi_{\mathcal{Q}_\nuv^*}\|_{p_0}
\lc \Lambda_{p_0,q,r}(N,\rho) (\rho N^\eps)^{2 d(\frac 1{p_0}-\frac 1p)}
\prod_{i=1}^2\|f_i\chi_{\mathcal{Q}_\nuv^*}\|_{p}
\end{align*}
We use  the Schwarz inequality, the embedding $\ell^p\subset \ell^q$, $p\le q$,
 and the fact that every $x$ is contained in only $O(N^{\eps d})$ of the cubes $\cQ_\nuv^*$ to get
\begin{align*}
\sum_\nuv\prod_{i=1}^2\|f_i\chi_{\mathcal{Q}_\nuv^*}\|_{p}^{q/2}
\le
\prod_{i=1}^2
\Big(\sum_\nuv\|f_i\chi_{\mathcal{Q}_\nuv^*}\|^{q}_{p}\Big)^{1/2}
\lc N^{\eps d}
\prod_{i=1}^2\|f_i\|_p^{q/2}\,.
\end{align*}
Combining the previous two estimates we bound
\Be\label{Inubound}(\,\sum_\nuv I_\nuv^{q/2}\,)^{2/q}\lc  N^{2d \eps(\frac{1}{p_0}-\frac 1p+\frac 1q)}
 \rho^{2d(\frac{1}{p_0}-\frac 1p)}
\Lambda_{p_0,q,r}(N,\rho)
\prod_{i=1}^2\|f_i\|_p.\Ee

We use very crude estimates to handle the remaining three terms
which can to be dominated by
$C_{M,\eps}(N^\eps \rho)^{-M}\|f_1\|_p \|f_2\|_p,$ which finishes the proof since
$\Lambda_{p_0,q,r}(N,\rho)\gc N^{\frac 2q -\frac 2r}$ by \eqref{lowertriv}.

We  only give the argument to bound $\sum_\nuv I\!I_\nuv^{q/2}$ as
the other terms  are handled similarly. By the Schwarz inequality we estimate $\sum_\nuv I\!I_\nuv^{q/2}$ by
\Be\label{IIsplit}
\Big(\sum_\nuv \big\|\chi_{\cP_\nuv}S_1[f_1\chi_{\bbR^d\setminus
\cQ_\nuv^*}]\big\|^q_{L^q(\bbR^d;L^r[0,\rho])}\Big)^{1/2} \Big(\sum_\nuv
\big\|S_2[f_2\chi_{\cQ_\nuv^*}]\big\|_{L^q(\bbR^d;L^r[0,\rho])}^q\Big)^{1/2}\,. \Ee
For the second factor we use a  wasteful bound,
namely that the $L^p\to L^q(\bbR^d; L^r[0,\rho])$ operator norm of
$S_2$ is $O(\rho^{1/r}N^{d})$. Consequently, the second factor in
\eqref{IIsplit} can be bounded by $C\rho^{q/(2r)}
N^{d(\eps+q/2)}\|f_2\|_p^{q/2}$.

We consider the first factor in \eqref{IIsplit} and write
 $S_1f(x,t)= \cK_t*f(x)$ where
$$
\cK_t(y)=\frac{1}{(2\pi)^{d}}
\int_{\mathbb{R}^d}
\chi(\xi-Ne_1)e^{-it|\xi|^2+i\inn{y}{\xi}}\,d\xi
$$
with $\chi \in C_c^\infty$ equal to 1 on the ball of radius $2d$ centred at the origin.
Integration by parts yields that for every $t\in [0,\rho]$
$$|\cK_t(y)|\le C_{M}|y-2tNe_1|^{-M}\quad\text{ if }\ |y-2tNe_1|\ge 4d\rho.$$
 Let $\fc_\nuv$ be the centre of
$\cQ_\nuv^*$. If $x-y\in \bbR^d\setminus \cQ^*_\nuv$ and $(x,t)\in
\cP_\nuv$, then $|x-y-\fc_\nu|\ge 10d \rho N^\eps$,
$|x-2tNe_1-\fc_\nu| \le 2d \rho N^\eps$,  and therefore also $|y-2t Ne_1|
\ge 8d \rho N^\eps$. Thus for this choice of $(x,t)$ and $y$ we have
$$
\big|S_1[f_1 \chi_{\bbR^d\setminus \cQ_\nuv^*}]\big| \lc (\rho
N^{\eps})^{-M+d+1} \int_{\substack{|y-2tNe_1|\ge 8d\rho  N^\eps}}
\frac{|f_1(x-y)|}{|y-2tNe_1|^{d+1}}dy
$$ and the integral is  bounded by
$(\rho N)^{d+1}\int (1+|y|)^{-d-1} |f_1(x-y)|dy$. Here we use
$\rho> 1$. Now let $\cQ_\nuv^{**}$ be the cube of sidelength
$\rho(2+N)$ centred at $\fc_\nu$; then
$\cQ_\nuv^{**}\times[0,\rho]$ contains $\cP_\nuv$. Letting
$\cC_{\rho,N}:=\rho^{1/r}(\rho N^{\eps})^{-M_1+d+1}(\rho N)^{d+1}$,
we have
\begin{align*}&\sum_\nuv
 \big\|\chi_{\cP_\nuv}S_1[f_1\chi_{\bbR^d\setminus \cQ_\nuv^*}]\big\|_{L^q(\bbR^d;L^r[0,\rho])}^q
\lc\cC_{\rho,N}^q\sum_\nuv \int_{\cQ_\nuv^{**}}\Big| \int
\frac{|f_1(x-y)|}{(1+|y|)^{d+1}}dy\Big|^q dx
\end{align*}
which is $\lc\cC_{\rho,N}^q(\rho N)^{(d+1)}
\|f_1\|_p^q$; here one uses Young's inequality and the fact that
each $x\in \bbR^d$ is contained in at most $O((\rho N)^{d+1})$ of the cubes $\cQ_\nu^{**}$.
Collecting the estimates yields  the crude bound
$$\sum_\nuv I\!I_\nuv^{q/2}
\le C_{M} (\rho N^{\eps})^{-M} (\rho N)^{10dq} \|f_1\|_p^{q/2}
\|f_2\|_p^{q/2},$$
and we conclude by  choosing $M$ sufficiently large.
\end{proof}

\begin{proof}[Proof of Lemma \ref{billemma}]
For $j\ge 0$, we write
$$A(j,\la ):= 2^{2j(\frac{2}{r}-d(\frac 1p-\frac 1q))} \sup_{2^{j-1}\le N\le 2^{j+1}}
\Lambda_{p,q,r}(N, C\la ^22^{-2{j+1}}).
$$
Define $T=U\psi(D) $,
and thus  $U\psi(\tfrac{D}{\la} )f(x,t)= T[f(\la ^{-1}\cdot)](\la x,
\la ^2t)$. By scaling, \Be\label{firstscaling} \|U\psi\big(\tfrac
D\la \big)\|_{L^p\to {L^q(\bbR^d;L^r[0,1]})} = \la
^{-\frac2r+d(\frac 1p-\frac 1q)} \|T\|_{L^p\to L^q(\bbR^d;L^r[0,\la
^2])}\, , \Ee so that the statement of the lemma is an immediate
consequence of \Be\label{Tbound}\|T\|_{L^p\to L^q(\mathbb
R^d;L^r[0,\la ^2])}
\lc\Big(\la^{\frac{4}q-2d(\frac1p-\frac1q)}+\sum_{1\le 2^j\le \la}
A(j,\la) \Big)^{1/2}. \Ee

Now by scaling we have that \Be \label{separationrescaled} \big\|Tf_1\,
Tf_2\big\|_{L^{q/2}(\mathbb R^d ; L^{r/2}[0,\la^2])} \lc
A(j,\la) \prod_{i=1}^2\|f_i\|_p, \Ee whenever $\widehat f_1$ and
$\widehat f_2$ are supported in a $2^{-j+1}$ ball, contained
in $\{\xi:1/2<|\xi|\le 2\}$,  and their supports are
$2^{-j}$-separated.  We will also require the following
simpler estimates
 \Be \label{noseparation} \big\|Tf_1\,
Tf_2\big\|_{L^{q/2}(\mathbb R^d ; L^{r/2}[0,\la^2])} \lc
\la^{\frac{4}{q}-2d(\frac{1}{p}-\frac{1}{q})}\prod_{i=1}^2\|f_i\|_p, \Ee
 whenever $\widehat f_1$ and
$\widehat f_2$ are supported in an ball of radius $\la^{-1}$, contained
in $\{\xi:1/2<|\xi|\le 2\}$. By the Schwarz inequality, this follows from $\big\|Tf_1 \big\|_{L^{q}(\mathbb R^d ; L^{r}[0,\la^2])}\lc
\la^{\frac2q-d(\frac1p-\frac1q)}\|f_1\|_p$. Let $t\mapsto \varpi(t)$ be a Schwartz function which is
positive on $[0, 4d]$ and whose Fourier transform is supported in $[-1,1]$.
 By scaling and rotation this would follow from \Be \label{triviallinear}
\big\| \varpi Tf\big\|_{L^{q}(\mathbb R^d ;L^{r}(\R))}
\lc \la^{\frac{2}q-\frac{2}r}\|f\|_p\Ee whenever
$\widehat f$ is supported in  $\{\xi: |\xi-\la e_1|\le 2d\}$. By
a change of variables and trivial estimates it is easy to see
\eqref{triviallinear} for $1\le p\le q=r\le \infty$.
The estimate for $r>q$  follows by applying Bernstein's inequality in $t$
since the temporal Fourier transform of $\varpi Tf$ is contained in
$\{ s :  s \sim \la^2\}$.

We now argue  similarly as in \cite{tavave}. Write $\|T
f\|_{L^{q}(\mathbb R^d ;L^r[0,\la^2])}^2=\|Tf\,Tf\|_{L^{q/2}(\mathbb
R^d ;L^{r/2}[0,\la^2])}.$
For each $j$, $1\le 2^{j}\le 2\la$,  we  tile $\bbR^d$
 with dyadic cubes $ s _{\ell}^j=\prod_{i=1}^d [2^{-j}\ell_i, 2^{-j} \ell_{i+1})$
of sidelength $2^{-j}$, indexed by $\ell \in \bbZ^d$. For  $j$,
$1\le 2^{j}\le \la$, we write $\ell\sim_j\tell$ if $ s _{\ell}^j$ and
$ s _{\tell}^{j}$ have adjacent parents, but are not adjacent. When
$\la< 2^{j}\le 2\la$, we mean  by
$\ell\sim_j\tell$
that the distance between $ s _{\ell}^j$
and $ s _{\tell}^{j}$ is $\lesssim \la^{-1}$. Then, we then can write
for every $(\xi,\eta)\in \bbR^d$, with $\xi\neq \eta$,
\Be\label{partofunity}
\sum_{1\le 2^{j}\le 2\la} \sum_{\substack {(\ell,\tell)\\
\ell\sim_j\tell}}\chi_{ s _\ell^j}(\xi) \chi_{ s _\tell^j}(\eta)=1
\Ee

Define $P^j_\ell$ by $\widehat {P^j_\ell f}=
\chi_{ s _{\ell}^j}\widehat{f}$; then
the operators $P^j_\ell$ are bounded on $L^p$, $1<p<\infty$,
 with operator norms independent of $\ell$ and $j$.
For any
Schwartz function $f$
we have by \eqref{partofunity}
$$
[T f(x,t)]^2= \sum_{1\le 2^{j}\le 2\la} \sum_{(\ell,\tell):
\ell\sim_j\tell} T P_{\ell}^j f(x,t)\,T P_{\tell}^{j}f(x,t)
$$
Let $\vphi\in C^\infty_c$ be supported in $[-1,1]^d$, satisfying
 $\sum_{\fz\in\bbZ^d}\vphi(\xi-\fz)=1$ for all $\xi\in \bbR^d$. Define
$\widetilde P_\fz^j$ as acting on $L^a(L^b)$ functions by
$\widehat {\widetilde P_\fz^j G}(\xi,t)= \vphi(2^j\xi-\fz) \widehat G(\xi,t)$.
We use the inequality
\Be\label{qr-orth}
\Big\|\sum_{\fz} \widetilde P_{\fz}^j
G_\fz\Big\|_{L^a(L^b)}\le C
\big\|\{G_{\fz}\}\big\|_{\ell^{a}(L^a(L^b))},\quad
1\le a\le 2,\quad a\le b\le a',
\Ee
The constant $C$ in \eqref{qr-orth} is independent of $j$.
The inequality follows from  Plancherel's theorem in the case $a=b=2$,
and from an application of Minkowski's inequality in the case $a=1$, $1\le b\le \infty$. The intermediate cases follow by interpolation.
Note that for any $j$ and any $\fz\in \fZ^d$ the number of pairs $(\ell, \tell)$ with $\ell\sim_j \tell$ for which
$\widetilde P_{\fz}^j [T P_{\ell}^j f \,T P_{\tell}^{j}f]\neq 0$ is uniformly bounded
(independent of  $j$, $\fz$, $f$). Thus inequality
\eqref{qr-orth} applied with $a=q/2$, $b=r/2$,
implies
\begin{align}\label{uno}\begin{split}
\|T f\|^2_{L^{q}(L^r[0,\la^2])}\lc \sum_{1\le 2^{j}\le
2\la}\Big(\sum_{{\ell}\sim_j{\tell}}\|T P_{\ell}^jf\,\,T
P_{\tell}^{j}f\|^{q/2}_{L^{q/2}(L^{r/2}[0,\la^2])}\Big)^{2/q};
\end{split}\end{align}
here we use that
$1\le q/2\le r/2\le (q/2)'$
(i.e.  $q\le r\le \frac{2q}{q-2}$ which implies that $q/2\le 2$.)

Now by \eqref{separationrescaled} and \eqref{noseparation} the right
hand side of \eqref{uno} is dominated by a constant times
\begin{align*}
 \sum_{1\le 2^{j}\le \la} A(j,\la) \Big(\sum_{{\ell}\sim_j{\tell}}\|
P_{\ell}^jf\|_p^{q/2}\|P_{\tell}^{j}f\|^{q/2}_{p}\Big)^{2/q}+\la^{\frac{4}{q}-2d(\frac{1}{p}-\frac{1}{q})}\Big(\sum_{{\ell}\sim_{j_0}{\tell}}\|
P_{\ell}^{j_0}f\|_p^{q/2}\|P_{\tell}^{j_0}f\|^{q/2}_{p}\Big)^{2/q}\\\lc
\la^{\frac{4}{q}-2d(\frac{1}{p}-\frac{1}{q})}\Big(\sum_{\ell}\|
P_{\ell}^{j_0}f\|_p^{q}\Big)^{2/q}+\sum_{1\le 2^{j}\le \la}
A(j,\la) \Big(\sum_{\ell}\| P_{\ell}^jf\|_p^{q}\Big)^{2/q}.
\end{align*}
Here $j_0$ is the integer such that $\la<2^{j_0}\le 2\la$, and  we have
used the Schwarz inequality and the fact that for each $(j,\ell)$
the number of $\tell$ with $\ell\sim_j\tell$ is uniformly bounded.
Since $2\le p\le q$, we also have
$$\Big(\sum_{\ell}\| P_{\ell}^jf\|_p^{q}\Big)^{1/q}\lc \|f\|_p,$$
and thus we have shown \eqref{Tbound}.
\end{proof}


\begin{thebibliography}{1}


\bibitem{BCSS} W. Beckner, A. Carbery, S. Semmes, and F. Soria,
\emph{A note on the restriction of the Fourier Transform to
spheres,} Bull. London. Math. Soc. {\bf 21} (1989), 394--398.

\bibitem{boPrinc} J. Bourgain, \emph{Some new estimates on oscillatory
integrals,} Essays on Fourier analysis in honor of Elias M. Stein, 83–112,
Princeton Math. Ser., 42, Princeton Univ. Press, Princeton, NJ, 1995.

 \bibitem{B} \bysame, {\it On the   Schr\"odinger maximal function in higher dimension},  arXiv:1201.3342.

\bibitem{bogu} J. Bourgain and L. Guth, \emph{Bounds on oscillatory integral operators based on multilinear estimates,} Geom. Funct. Anal. {\bf 21} (2011), 1239--1295.

\bibitem{car} A. Carbery
\emph{Restriction implies Bochner--Riesz for paraboloids,} Math. Proc.
 Cambridge Philos. Soc. {\bf 111} (1992), no. 3, 525--529.

\bibitem{camax} L. Carleson, \emph{Some analytic problems related to statistical mechanics,}  Euclidean harmonic analysis,   5--45, Lecture Notes in Math., {779}, Springer, Berlin, 1980.



\bibitem{casj} L. Carleson and P. Sj\"olin,
   \emph{Oscillatory integrals and a multiplier problem for the disc,}
   Studia Math. {\bf 44} (1972), 287--299.

\bibitem{dake} B.E.J. Dahlberg\ and\ C.E. Kenig, \emph{A note on the almost everywhere behavior of solutions to the Schr\"odinger equation,}  Harmonic analysis, 205--209, Lecture Notes in Math., {908}, Springer, Berlin.

\bibitem{fe} C. Fefferman, \emph{Inequalities for strongly singular convolution operators,} Acta Math. {\bf  124}  (1970),  9--36.

\bibitem{F} \bysame,
\emph{The multiplier problem for the ball,} Ann. of Math. (2) {\bf 94} (1971), 330--336.


\bibitem{fest} C. Fefferman and E. Stein,  \emph{$H^p$ spaces of several variables,}
Acta Math.  {\bf 129}  (1972), no. 3-4, 137--193.


\bibitem{give}  J. Ginibre and G. Velo, \emph{The global Cauchy problem for the nonlinear Schrodinger equation revisited},     Ann. Inst. H. Poincare Anal. Non Lineare {\bf 2} (1985), 309-327.


\bibitem{hoer1} L. H\"ormander, \emph{Estimates for translation invariant operators in $L^{p}$ spaces,}  Acta Math.  {\bf 104} (1960), 93--140.

\bibitem{hoer2} \bysame, \emph{Oscillatory integrals and multipliers on $FL^{p}$,}  Ark. Mat. {\bf 11} (1973), 1--11.


\bibitem{keta} M. Keel and T. Tao, \emph{Endpoint Strichartz estimates,} Amer. J. Math. {\bf 120} (1998), 955--980.

\bibitem{keich} U. Keich, \emph{
On $L^p$ bounds for Kakeya maximal functions and the Minkowski dimension in ${\bbR}^2$,}
Bull. London Math. Soc. {\bf 31} (1999), no. 2, 213–-221.


\bibitem{KPV} C.E. Kenig, G. Ponce\ and\ L. Vega, \emph{Oscillatory integrals and regularity of dispersive equations}, Indiana Univ. Math. J. {\bf 40} (1991), no.~1, 33--69.

\bibitem{le} S. Lee, \emph{Improved bounds
for Bochner--Riesz and maximal Bochner--Riesz operators,} Duke Math.
J.  {\bf 122}  (2004), 205--232.


\bibitem{le1} \bysame, \emph{On pointwise convergence of the solutions to Schr\"odinger    equations in $\R^2$,} Int. Math. Res. Not., (2006), Art. ID 32597, 21.


\bibitem{lerose} S. Lee, K.M. Rogers and A. Seeger,
\emph{Improved bounds for Stein's square functions,} {\tt  arXiv:1012.2159}, {Proc. London. Math. Soc.}, doi:10.1112/plms/pdr067.




\bibitem{lerova} S. Lee, K.M. Rogers and A. Vargas, \emph{An endpoint space--time estimate for the Schr\"odinger equation,} Adv. Math. {\bf 226} (2011), 4266--4285.



 \bibitem{mi} A. Miyachi, \emph{On some singular Fourier multipliers,}
{J. Fac. Sci. Univ. Tokyo Sect. IA Math.,} {\bf 28} (1981),
{267--315}.




\bibitem{mss} G. Mockenhaupt, A. Seeger and C. D. Sogge.
\emph{Local smoothing of Fourier integral operators and Car\-leson-Sj\"o\-lin estimates,} J. Amer. Math. Soc. {\bf 6} (1993), no. 1, 65–-130.


\bibitem{ms} D. M\"uller and  A. Seeger, \emph{Regularity properties of wave propagation on conic manifolds and applications to spectral multipliers,}  Adv. Math.  {\bf 161}  (2001),  no. 1, 41–130.



\bibitem{planch}
F. Planchon, \emph{Dispersive estimates and the 2D cubic NLS equation,}
J. Anal. Math. {\bf 86} (2002), 319–-334.

\bibitem{ro1}  K.M. Rogers,  \emph{A local smoothing estimate for the Schr\"odinger equation,} Adv. Math. {\bf 219} (2008), no.
6, {2105--2122}.

\bibitem{ro}  \bysame ,  \emph{Strichartz estimates via the Schr\"odinger maximal operator,} Math. Ann. {\bf 343} (2009), no.3, {604--622}.

\bibitem{rose}  K.M. Rogers and A. Seeger,  \emph{Endpoint maximal and smoothing estimates for Schr\"odinger equations,} J. Reine Angew. Math.,
{\bf  640}  (2010), 47-66.

\bibitem{RV} A. Ruiz\ and\ L. Vega, On local regularity of Schr\"odinger equations, Internat. Math. Res. Notices {\bf 1993}, no.~1, 13--27.



\bibitem{sj} P. Sj{\"o}lin,
\emph{Regularity of solutions to the Schr\"odinger equation,}
  {Duke Math. J.}, {\bf 55} (1987),   {699--715}.


\bibitem{sj2} \bysame, \emph{A counter-example concerning maximal estimates for solutions to equations of Schr\"odinger type}, Indiana Univ. Math. J. {\bf 47} (1998), no.~2, 593--599.

\bibitem{steinbook} E.M. Stein, Harmonic analysis: real-variable methods, orthogonality, and oscillatory integrals. With the assistance of Timothy S. Murphy. Princeton Mathematical Series, 43. Monographs in Harmonic Analysis, III. Princeton University Press, Princeton, NJ, 1993.


\bibitem{str} R. Strichartz, \emph{Restrictions of Fourier transform to quadratic surfaces and decay of
solutions of Wave Equations,} Duke Math. J. \textbf{44} (1977), 705-714.

\bibitem{ta} T. Tao,
\emph{A sharp  bilinear restriction estimate for paraboloids,} Geom.
Funct. Anal. {\bf 13} (2003), 1359--1384.

\bibitem{tabook} T. Tao, \emph{Nonlinear dispersive equations,} Local and global analysis. CBMS 106, eds: AMS, 2006.


\bibitem{tava}T.  Tao and A.  Vargas, \emph{A bilinear approach to cone multipliers. II. Applications,} Geom. Funct. Anal. {\bf 10}  (2000), no. 1, 216--258.

\bibitem{tavave}
T. Tao, A. Vargas and L. Vega. \emph{A bilinear approach to the
restriction and Kakeya conjectures,} J. Amer. Math. Soc.  {\bf 11}
(1998), 967--1000.


\bibitem{ve} L. Vega,
\emph{Schr\"odinger equations: pointwise convergence to the initial
data,} Proc. Amer. Math. Soc.  {\bf 102}  (1988),  no. 4, 874--878.


\end{thebibliography}
\end{document}